\documentclass{amsart}
\usepackage{amsmath,amsthm}
\usepackage{amsfonts,amssymb}

\usepackage{enumerate}
\usepackage{graphicx}
\usepackage{pinlabel}

\newtheorem{thm}{Theorem}[section]
\newtheorem{cor}[thm]{Corollary}
\newtheorem{lem}[thm]{Lemma}
\newtheorem{prop}[thm]{Proposition}

\newtheorem{defn}[thm]{Definition}

\theoremstyle{remark}
\newtheorem{rem}[thm]{Remark}

\numberwithin{equation}{section}

\def\bell{{\boldsymbol{\large {\ell}}}} 
\def\bnu{{\boldsymbol{\large {\nu}}}}

 \def\b{{\beta}}
 
 \def\k{{\kappa}}

 \def\s{{\sigma}}
 \def\la{{\langle}}
 \def\ra{{\rangle}}

 \def\sB{{\mathsf B}}
 \def\sC{{\mathsf C}}
 
 \def\sH{{\mathsf H}}
 \def\sK{{\mathsf K}}
 \def\sM{{\mathsf M}}
 \def\sP{{\mathsf P}}
 \def\sQ{{\mathsf Q}}

 \def\cb{{\mathbf c}}

 \def\pb{{\mathbf p}}

 \def\xb{{\mathbf x}}
 \def\yb{{\mathbf y}}

 \def\CB{{\mathcal B}} 
 \def\cD{{\mathcal D}}

 \def\CI{{\mathcal I}}
 
 \def\CL{{\mathcal L}}
  \def\cL{{\mathcal L}}
 \def\cM{{\mathcal M}}

 \def\CV{{\mathcal V}}

 \def\NN{{\mathbb N}}

 \def\RR{{\mathbb R}}
 \def\SS{{\mathbb S}}
 \def\cS{{\mathcal S}}
 \def\ZZ{{\mathbb Z}}

\def\dim{\operatorname{dim}}

\newcommand{\wh}{\widehat}

\def\nut{\tilde \nu}

\newcommand{\pd}{\partial}
\newcommand{\fsh}{\triangle}
\newcommand{\bsh}{\nabla}

\newcommand{\ellt}{\tilde{\ell}}
\newcommand{\Nt}{\tilde{N}}
\newcommand{\fAd}{\mathfrak{A}_{d+1}}
\newcommand{\fGd}{\mathfrak{G}_{d+1}}

\begin{document}

\title[Hahn polynomials on polyhedra and quantum integrability]
{Hahn polynomials on polyhedra and quantum integrability}

\date{February 3, 2020}

\author{Plamen~Iliev}
\address{P.~Iliev, School of Mathematics, Georgia Institute of Technology, 
Atlanta, GA 30332--0160, USA}
\email{iliev@math.gatech.edu}
\thanks{The first author is partially supported by Simons Foundation Grant \#635462.}

\author{Yuan~Xu}
\address{Y.~Xu, Department of Mathematics, University of Oregon, Eugene, 
OR 97403--1222, USA}
\email{yuan@uoregon.edu}
\thanks{The second author is partially supported by NSF grant \#1510296.}


\begin{abstract}
Orthogonal polynomials with respect to the hypergeometric distribution on lattices 
in polyhedral domains in $\RR^d$, which include hexagons in $\RR^2$ and truncated tetrahedrons in $\RR^3$, are defined 
and studied. The polynomials are given explicitly in terms of the classical one-dimensional Hahn polynomials. They are also characterized as common
eigenfunctions of a family of commuting partial difference operators. These operators provide symmetries for a system that 
can be regarded as a discrete extension of the generic quantum superintegrable system on the $d$-sphere. Moreover, the discrete 
system is proved to possess all essential properties of the continuous system. In particular, the symmetry operators for 
the discrete Hamiltonian define a representation of the Kohno-Drinfeld Lie algebra on the space of orthogonal polynomials, 
and an explicit set of $2d-1$ generators for the symmetry algebra is constructed. Furthermore, other discrete quantum superintegrable systems, which extend the quantum harmonic oscillator, are obtained by considering appropriate limits of the parameters.
\end{abstract}

\maketitle
 
\tableofcontents
  
\section{Introduction}\label{se1}

Hahn polynomials of $d$ variables are usually defined as discrete orthogonal polynomials on lattices in either a product 
domain or a simplex in $\RR^d$. We shall define and study such polynomials on polyhedral domains in $\RR^d$ that include 
hexagons in $\RR^2$ and truncated tetrahedrons in $\RR^3$ (see Figure 1). 

\begin{figure}[htb]
\labellist
\small\hair 1.5pt
 \pinlabel {$\ell_1$} [ ] at 213 31
 \pinlabel {$\ell_2$} [ ] at 28 284
 \pinlabel {$N-\ell_3$} [ ] at 20 146
 \pinlabel {$N-\ell_3$} [ ] at 143 30
 \pinlabel {$N$} [ ] at 351 34
 \pinlabel {$N$} [ ] at 28 352
\endlabellist
\centering
\includegraphics[scale=0.48]{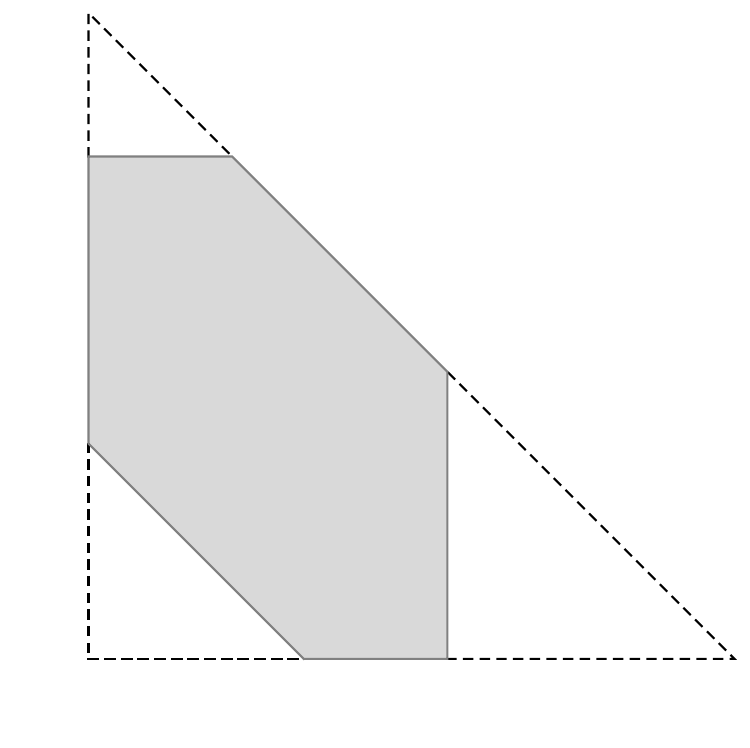}
\quad
\labellist
\small\hair 1.5pt
 \pinlabel {$\ell_1$} [ ] at 30 103
 \pinlabel {$\ell_2$} [ ] at 306 98
 \pinlabel {$\ell_3$} [ ] at 148 377
 \pinlabel {$N-\ell_4$} [ ] at 104 153
 \pinlabel {$N-\ell_4$} [ ] at 195 156
 \pinlabel {$N-\ell_4$} [ ] at 153 225
\endlabellist
\centering
\includegraphics[scale=0.48]{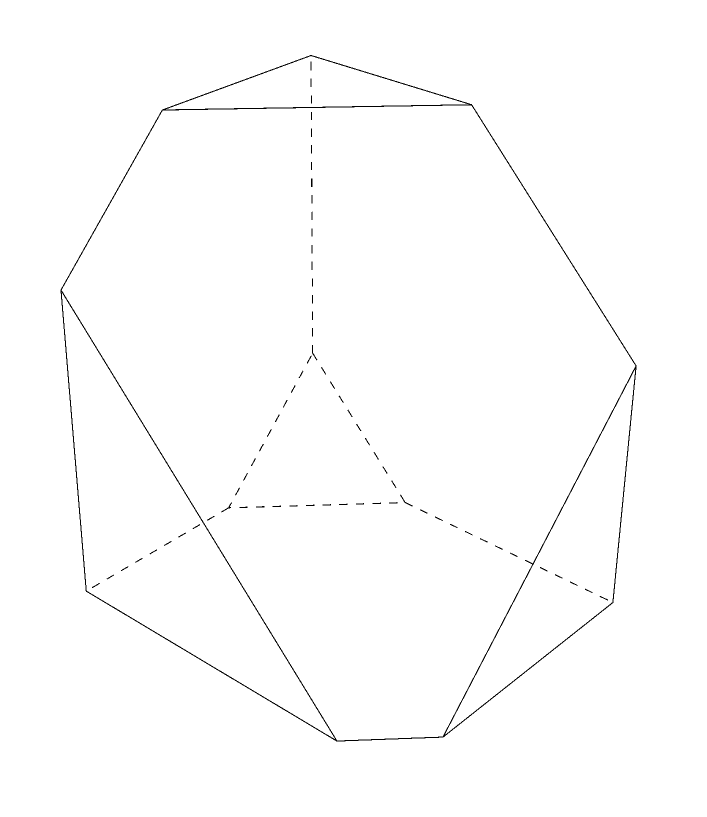}
\vspace{-0.2in}
\caption{Domain $V_{\ell,N}^d$ for $d =2$ and $d=3$}
\end{figure}
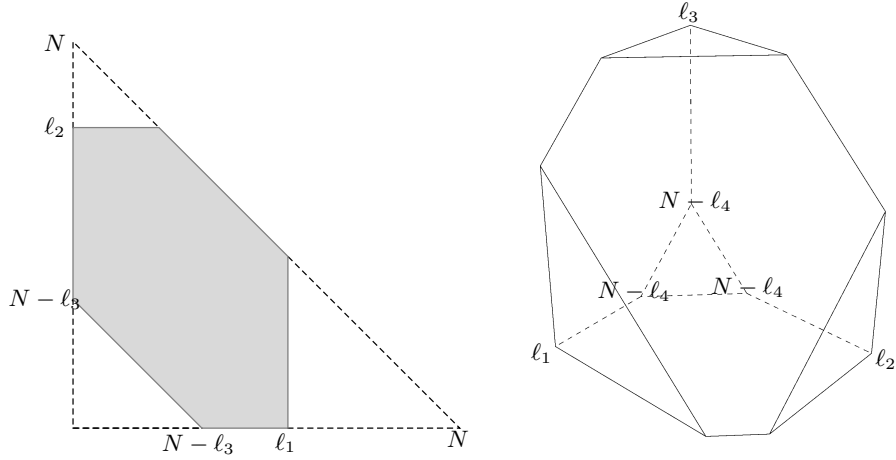

More specifically, for $N \in \NN$ and 
$\ell_i \in \NN_0$ with $\ell_i \le N$, $1 \le i \le d+1$, we consider the discrete polyhedral domain 
\begin{equation}\label{eq:poly-domain}
   V_{\ell, N}^d: = \left \{x \in \NN_0^d: 0 \le x_i \le \ell_i, \, 1 \le i \le d, \, \hbox{and}\, N-\ell_{d+1} \le |x| \le N\right\},
\end{equation}
where $|x|=x_1+\cdots +x_d$.
Let $(a)_n := a(a+1)\ldots (a+n-1)$ be the Pochhammer symbol. On the domain $V_{\ell,N}^d$, we define a positive weight 
function 
\begin{equation}\label{eq:poly-weight}
  \sH_{\ell,N}(x) =  \frac{N!}{(-|\ell |)_N}\prod_{i=1}^d \frac{(-\ell_i)_{x_i}}{x_i!} \frac{(-\ell_{d+1})_{N-|x|}}{(N-|x|)!}, 
\end{equation}
which corresponds to the hypergeometric distribution on the discrete polyhedral domain. Our goal is to 
construct orthogonal polynomials with respect to the inner product 
\begin{equation} \label{eq:Hahn-ipd}
  \la f, g\ra_{\ell,N} =  \sum_{x \in V_{\ell,N}^d} f(x) g(x)  \sH_{\ell,N} (x),
\end{equation}
normalized so that $\la 1,1\ra_{\ell,N} =1$, and study partial difference operators that have these orthogonal polynomials as common 
eigenfunctions.

Let $\Pi_N^d$ denote the space of polynomials of total degree at most $N$ in $d$ variables. Let $\CI(V_{\ell,N}^d)$ denote 
the ideal of polynomials that vanish on $V_{\ell,N}^d$. It is known that the space of orthogonal polynomials, denoted by
$\Pi_{\ell,N}^d$, with respect to $\la \cdot,\cdot\ra_{\ell,N}$ can be identified with $\RR[x_1,\ldots,x_d]/\CI(V_{\ell,N}^d)$ and
there is a lattice set $\Lambda_{\ell,N}^d$ such that every polynomial in the quotient space can be written as
\begin{equation} \label{Lambda-set}
  \sP(x) = \sum_{\nu \in \Lambda_{\ell,N}^d} c_\nu x^\nu.  
\end{equation}
It is surprisingly difficult to determine $\Lambda_{\ell,N}^d$ in our case, which in turn causes difficulty in the construction of a basis for $\Pi_{\ell,N}^d$.  
However, it turns out, that a basis of $\Pi_{\ell,N}^d$ can be given explicitly in terms of the classical Hahn polynomials, which 
we shall call Hahn polynomials on the polyhedron. Furthermore, this basis is uniquely determined by a family of commuting 
partial difference operators that have the Hahn polynomials on the polyhedron as common eigenfunctions. 

When $\k_i=-\ell_i-1$ are real numbers such that $\k_i>-1$, the weight function \eqref{eq:poly-weight} becomes the Dirichlet multinomial distribution on the discrete simplex
\begin{equation}\label{eq:VNd}
   V_N^d: = \{\nu \in \NN_0^d: |\nu| \le N\}. 
\end{equation}
In this case, the orthogonal polynomials are the Hahn polynomials of several variables that have been studied extensively; 
see, for example, \cite{Du,GI, GS, IX07, IX17, KM, X15} and the references therein. In particular, an explicit orthogonal basis 
was introduced in \cite{KM} and reformulated later in \cite{IX07} in its present form given here (see \eqref{eq:Hahn-simplex}). 
The second-order difference operators that have orthogonal polynomials as eigenfunctions are classified 
in \cite{IX07}, where it was realized that the classification includes not only the Hahn polynomials on $V_N^d$, but also 
orthogonal polynomials on discrete product domains and more general domains that can be obtained form the discrete 
simplex by cutting off some, but not all, of its corners. The latter is equivalent to setting some, but not all, of the $\k_i$ to be 
negative integers. In all these cases, the set of the indices, $H_{\ell,N}^d$, of orthogonal polynomials can be identified with the 
lattice set $V_{\ell,N}^d$, on which the inner product is defined. 

Our setup in this paper corresponds to setting all $\k_i$ as negative integers. This case turns out to be very different from all 
previous cases in several fundamental aspects. The first major difference is that the index set $H_{\ell,N}^d$ and the lattice set
$V_{\ell,N}^d$ are no longer the same. In fact, they now have very different structures. What is still true is that these two sets 
have the same cardinality; that is, $|H_{\ell,N}^d| =|V_{\ell,N}^d|$. The proof of this fact turns out to be a highly
nontrivial problem of counting integer lattice points in a polyhedral domain. For $d =3$, the example in
Figure 2 gives an indication of what the task amounts to, where each cube in the left figure corresponds to one lattice point 
in $H_{\ell,N}^3$. The equation $|H_{\ell,N}^d| =| V_{\ell,N}^d|$ means that the two figures contain the same number of 
lattice points. 
\begin{figure}[htbp]  
\centering
 \includegraphics[width=2.5in]{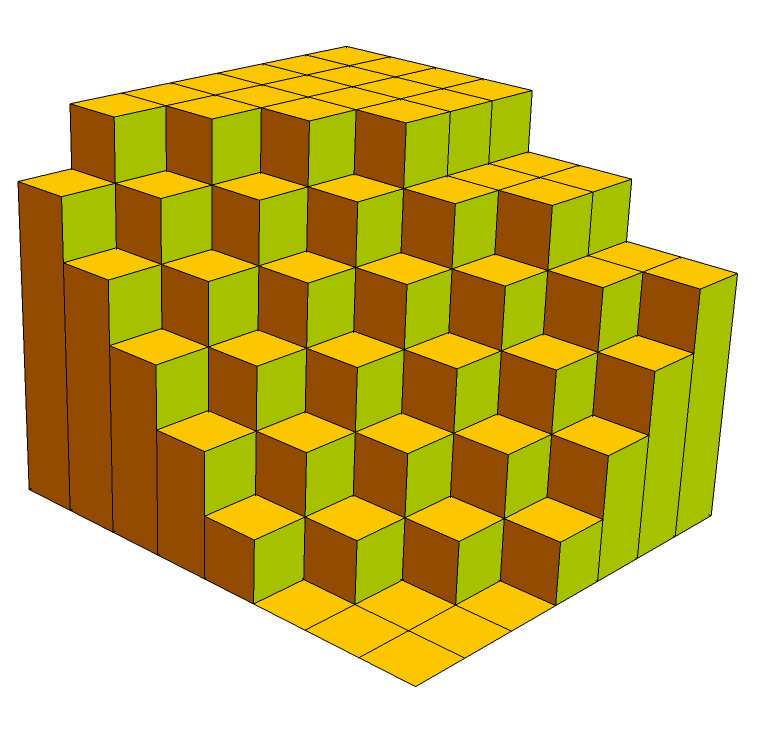} \includegraphics[width=2.4in]{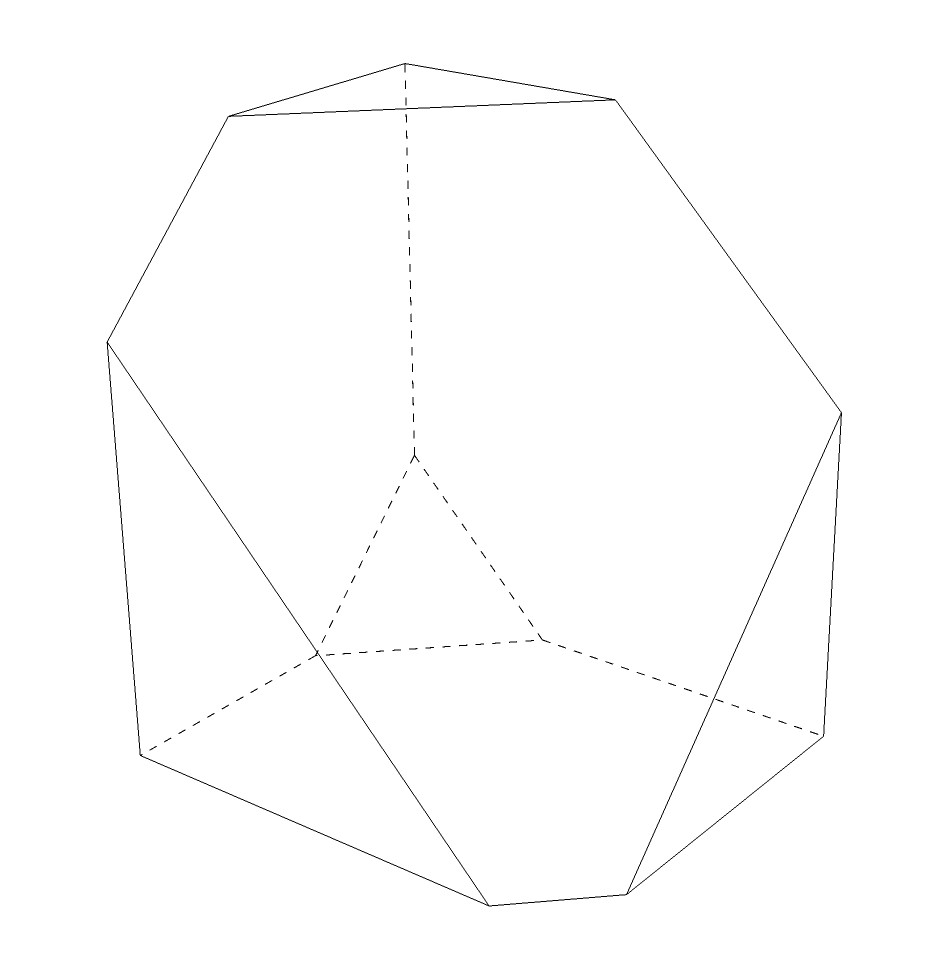}   
 \caption{$\ell=(6,7,5,8)$, $N = 10$. Left: lattice points in $H_{\ell,N}^3$. Right: boundary of $V_{\ell,N}^3$.} 
\end{figure}
Moreover, the complex structure of $H_{\ell,N}^d$ implies that there is no immediate description of the space $\Pi_{\ell,N}^d$ of 
all orthogonal polynomials in terms of a basis of monomials. However, there is a natural algebra of self-adjoint partial difference 
operators that acts on $\Pi_{\ell,N}^d$, and the index set can be characterized by an appropriate commutative subalgebra, which 
can be simultaneously diagonalized, providing an orthogonal basis of polynomials. 

The self-adjoint partial difference operators can be thought of as symmetries for a discrete quantum Hamiltonian system.
The latter can be regarded as an extension of the generic quantum superintegrable system on the sphere 
$\SS^{d}=\{z\in\RR^{d+1}: z_1^2+\cdots+z_{d+1}^2=1\}$, with Hamiltonian
\begin{equation}\label{1.6}
\mathcal{H}=\Delta_{\SS^d}+\sum_{k=1}^{d+1}\frac{b_k}{z_k^2},
\end{equation}
where $\Delta_{\SS^d}$ is the Laplace-Beltrami operator on the sphere and $\{b_k\}_{k=1,\dots,d+1}$ are parameters. 
The system on the sphere has been extensively studied in the literature as an important example of a second-order
superintegrable system, possessing $2d-1$ second-order algebraically independent symmetries, see for instance 
\cite{DGVV,I17,KMP,KMT} and the references therein. Recently, in \cite{I17}, it was shown that the symmetry operators for 
the Hamiltonian $\mathcal{H}$ define a representation of the Kohno-Drinfeld Lie algebra on the space of polynomials 
orthogonal with respect to the Jacobi weight function on the simplex, and an explicit set of $2d-1$ generators for the 
symmetry algebra was determined. It is known that the Jacobi polynomials on the simplex are generating functions 
for the Hahn polynomials on the discrete simplex, so one might expect that the discrete quantum Hahn system and the continuous quantum system on the sphere are closely related.  Nonetheless, it is surprising that essentially all properties of the continuous system on the sphere can be extended to 
the discrete system, as we show in this paper. 

Moreover, the Hahn polynomials on the simplex are on the very top of the classification in  \cite{IX07} of discrete orthogonal polynomials which are eigenfunctions of second order partial difference operators, and Krawtchouk, Meixner and Charlier polynomials of several variables can be regarded as limiting cases. Thus, we obtain several other discrete superintegrable systems via limits and describe their symmetry algebras. We also show that these systems extend the quantum harmonic oscillator and we derive its symmetry algebra and generators as a corollary. 

The paper is organized as follows. In the next section we introduce and study the Hahn polynomials on polyhedra, but 
postpone all technical proofs to later sections. Difference operators and quantum integrability will be studied in Section~\ref{se3}. In the
case when $d =2$, it is possible to say more about the correspondence between the index set and the lattice set, which is 
established in Section \ref{se4}. Finally, the counting of integer lattice points in the polyhedron $H_{\ell,N}^d$ is carried out 
in Section~\ref{se5}. 

\section{Hahn polynomials on polyhedra}\label{se2}

We start the section with the definition of Hahn polynomials on polyhedra and we formulate our main results for these
polynomials. The restriction  on admissible polyhedra is discussed in the second subsection. We will need the classical Hahn
polynomials with negative integer parameters, which we discuss in the third subsection, and we review Hahn polynomials 
of several variables and show how they are related to the framework of our study in the fourth subsection.  

\subsection{Hahn polynomials on polyhedra}\label{ss2.1}
For $N \in \NN$ and $1 \le i \le d+1$, let $\ell_i \in \NN$ and $\ell_i \le N$, and assume further that
\begin{equation} \label{eq:l-condition}
  \ell_i+ \ell_{j} \ge N, \qquad 1 \le i <j \le d+1. 
\end{equation}
Then the weight function $\sH_{\ell,N}$ defined in \eqref{eq:poly-weight} is positive on the discrete polyhedron $V_{\ell,N}^d$
defined in \eqref{eq:poly-domain}. When all inequalities in \eqref{eq:l-condition} are strict, the set is a hexagon for $d =2$ and 
a truncated tetrahedron, irregular in general, for $d =3$,  as shown in Figure 1. The domain degenerates when the inequalities in
\eqref{eq:l-condition} do not hold for one or more pairs of indices; see the discussion in the next subsection. 
 
For a finite set $S$, we denote its cardinality by $|S|$. It is easy to see that the 
number of lattice points in $V_{\ell,N}^d$ is exactly 
\begin{equation} \label{eq:dimension}
  |V_{\ell,N}^d| = \binom{N+d}{d} - \sum_{i=1}^{d+1} \binom{N-\ell_i+d-1}{d}.
\end{equation}

We consider orthogonal polynomials with respect to the inner product $\la \cdot,\cdot \ra_{\ell,N}$, defined in \eqref{eq:Hahn-ipd}. 
From the theory of orthogonal polynomials on general discrete sets \cite{X04}, we know that orthogonal polynomials on 
$V_{\ell,N}^d$ belong to the space 
$$
\Pi_{\ell,N}^d =  \RR[x_1,\ldots,x_d]/ \CI(V_{\ell,N}^d).
$$
To motivate this definition, note that if $\sQ \in \CI(V_{\ell,N}^d)$, then $\la \sQ, \sP \ra_{\ell,N} =0$ for all $\sP$, so it is natural 
to think of $\sQ$ as the zero vector in the space. From this definition, it is easy to see that
\begin{equation} \label{eq:OPdim}
  \dim \Pi_{\ell,N}^d = |V_{\ell,N}^d|. 
\end{equation}
Applying the Gram-Schmidt process on monomials allows us to generate a 
basis of orthogonal polynomials in $\Pi_{\ell,N}^d$. 
It also leads to the construction of a lattice set 
$\Lambda_{\ell,N}^{d}\subset \NN_0^{d}$ such that every 
polynomial in $\Pi_{\ell,N}^d $ can be uniquely written as a sum of monomials $x^\nu$ with $\nu \in \Lambda_{\ell,N}^d$ as in 
\eqref{Lambda-set}. However, all difference equations and recurrence relations for these polynomials should be considered as 
elements in $\Pi_{\ell,N}^d$, i.e. modulo the ideal $\CI(V_{\ell,N}^d)$.

\begin{thm}
Let $N \in \NN$ and $\ell_i \in \NN$. Then  
\begin{equation} \label{eq:ideal-basis}
 \CI(V_{\ell,N}^d) \cap \Pi_N^d = \bigoplus_{i=1}^d (-x_i)_{\ell_i+1} \Pi_{N-\ell_i-1}^d 
      \bigoplus (N-|x|)_{\ell_{d+1}+1}\Pi_{N-\ell_{d+1}-1}^d. 
\end{equation}
Furthermore, 
\begin{equation} \label{eq:OPspace}
      \Pi_N^d = \Pi_{\ell,N}^d \oplus ( \CI(V_{\ell,N}^d) \cap \Pi_N^d).
\end{equation}
\end{thm}

\begin{proof}
Since $\{ (-x_1)_{\nu_1}(-x_2)_{\nu_2}\cdots (-x_d)_{\nu_d}: |\nu|=N+1\}$ generates the polynomial ideal $\CI(V_N^d)$ and
$V_{\ell,N}^d \subset V_N^d$, we only need to consider $\CI(V_{\ell,N}^d) \cap \Pi_N^d$ and \eqref{eq:OPspace} follows immediately from \eqref{eq:ideal-basis}. Note next that the right-hand side of \eqref{eq:ideal-basis} is contained in the left-hand side, and therefore, the proof of \eqref{eq:ideal-basis} will follow if can show that these two spaces have the same dimension. 

Let $p_i \in \Pi_{N-\ell_i-1}^d$ be an arbitrary polynomial in $\Pi_{N-\ell_i-1}^d$. We show below that if 
\begin{equation}\label{eq:ind}
   \sum_{i=1}^d  (-x_i)_{\ell_i+1} p_i(x) + (N-|x|)_{\ell_{d+1}+1} p_{d+1}(x) = 0, \quad \forall x \in V_{\ell,N}^d, 
\end{equation}
then $p_i(x) = 0$ for $i =1,\ldots, d+1$. This will imply that the space in the right-hand side of \eqref{eq:ideal-basis} has dimension $\sum_{j=1}^{d+1}\binom{N-\ell_j-1+d}{d}$, which combined with \eqref{eq:dimension} and \eqref{eq:OPdim} will complete the proof.

Let $\nu \in V_1:=\{x \in \NN_0^d:  \ell_1 < x_1 \le N, |x| \le N\}$. It follows that
$(-v_i)_{\ell_i+1} =0$ for $i =2,\ldots,d$ and $(N-|v|)_{\ell_{d+1}+1} =0$, and therefore \eqref{eq:ind} implies that 
$(-v_1)_{\ell_1+1} p_1(v)=0$ for all $v \in V_1$. Let $y_1 = x_1 -\ell_1-1$ and $y_i = x_i$ for $i=2,\ldots d$. Note that $V_1$ is a translation of $\{y\in\NN_0^{d}: |y| \le N-\ell_1-1\}$ and therefore the polynomial interpolation on the lattice points of $V_1$ is unique. Since $p_1 \in \Pi_{N-\ell_1-1}^d$ and $p_1$ vanishes on 
$V_1$, it follows that $p_1(x) \equiv 0$. The same argument works for $i =2,3,\ldots, d+1$. 
\end{proof}

\begin{cor}
The ideal $\CI(V_{\ell,N}^d)$ is generated by the set 
$$
  \{(-x_i)_{\ell_i+1}, \,\, 1 \le i \le d\} \cup \{ (N-|x|)_{\ell_{d+1}+1}\} \cup 
   \{ (-x_1)_{\nu_1}(-x_2)_{\nu_2}\cdots (-x_d)_{\nu_d}: |\nu|=N+1\}.
$$
\end{cor}

The lattice set $\Lambda_{\ell,N}^d$ in \eqref{Lambda-set} is not uniquely determined by $V_{\ell,N}^d$. In fact, it is not 
clear how to identify this set in our setting. This makes the construction of an orthogonal basis for $\Pi_{\ell,N}^d$ difficult. It
turns out, however, that a basis of $\Pi_{\ell,N}^d$ can be given explicitly in terms of the classical Hahn polynomials of one variable,
defined by 
$$
  \sQ_n(x; a, b, N) = {}_3F_2\left( \begin{matrix} -n, \,  n+ a +b+1, \, -x \\ a+1, \, -N \end{matrix}; 1 \right),
$$
where $a, b$ are two real parameters, $N$ is a positive integer. 
To describe our orthogonal polynomials, we will need the following notations: For $y=(y_1,\ldots, y_{d}) 
\in \RR^{d}$ and $1 \le j \le d$, we define 
\begin{equation}\label{xsupj}
    \yb_j := (y_1, \ldots, y_j) \quad \hbox{and}\quad \yb^j := (y_j, \ldots, y_d), 
\end{equation}
and also define $\yb_0 := 0$ and $\yb^{d+1} := 0$. It follows that $\yb_d = \yb^1 = y$, 
and 
$$
   |\yb_j| = y_1 + \cdots + y_j,   \quad |\yb^j| = y_j + \cdots + y_d, \quad\hbox{and}\quad
   |\yb_0| = |\yb^{d+1}| = 0.
$$
For $\ell = (\ell_1,\ldots, \ell_{d+1})$, we have $\bell^j := (\ell_j, \ldots, \ell_{d+1})$ for $1 \le j \le d+1$. Furthermore,
\begin{equation*}
   a_j:=a_j(\kappa,\nu):= - |\bell^{j+1}| + 2 |\bnu^{j+1}| -1, \qquad 1 \le j \le d.
\end{equation*}
We now define Hahn polynomials on the polyhedron by 
\begin{align}\label{eq:Hahn-simplex}
\sQ_\nu(x;\ell, N) =& \frac{(-1)^{|\nu|}}{(-N)_{|\nu|}} 
\prod_{j=1}^d  \frac{(- \ell_j)_{\nu_j}}{(a_j+1)_{\nu_j}}(-N+|\xb_{j-1}|+|\bnu^{j+1}|)_{\nu_j}  \\ 
    & \times   \sQ_{\nu_j}(x_j; - \ell_j-1, a_j, N- |\xb_{j-1}|-|\bnu^{j+1}|). \notag
\end{align}

\begin{defn}\label{def:Hln}
For $N \in \NN$ and $\ell_i \in \NN$ satisfying \eqref{eq:l-condition}, define 
\begin{align*}
  H_{\ell,N}^d := \left \{   \nu \in \NN_0^d:  \right. & |\nu| \le N, \, |\nu| \le |\ell|-N,  \\
      & \left. \nu_j \le \ell_j, \,  2 |\bnu^{j+1}| \le |\bell^{j+1}|-\nu_j, \, 1 \le j \le d \right \}.
\end{align*}
\end{defn}
Note that $H_{\ell,N}^d$ is a set of lattice points in a polytope inside $V_N^d$ defined in \eqref{eq:VNd}.
The Hahn polynomials $\sQ(\cdot;\ell,N)$ on the polyhedron are defined when $\nu \in H_{\ell,N}^d$. 
More precisely, the following theorem holds: 

\begin{thm} \label{thm:OP-polyhedra}
Let $N \in \NN$ and let $\ell_i \in \NN$ satisfy \eqref{eq:l-condition}. Then 
\begin{enumerate}[\quad \rm (i)]
\item The polynomials $\sQ_\nu(\cdot; \ell,N)$ are orthogonal and satisfy 
$$
   \la \sQ_\nu(\cdot; \ell,N), \sQ_\mu(\cdot; \ell,N)  \ra_{\ell,N} = \sB_\nu(\ell, N) \delta_{\nu,\mu}
$$
for all $\nu,\mu \in H_{\ell,N}^d$, where 
\begin{align*}
\sB_\nu(\ell, N)  :=\frac{(-1)^{|\nu|}(-|\ell |)_{N+|\nu|}} {(-N)_{|\nu|} (-|\ell|)_N (-|\ell|)_{2|\nu|}} 
   \prod_{j=1}^d \frac{(-\ell_j+a_j)_{2\nu_j}  (-\ell_j)_{\nu_j} \nu_j! }
       { (-\ell_j+a_j)_{\nu_j} (a_j+1)_{\nu_j}}.
\end{align*}
\item The set $\{\sQ_\nu(\cdot ; \ell ,N): \nu \in H_{\ell,N}^d \}$ is a basis of $\Pi_{\ell,N}^d$. In particular,  
$$
 |H_{\ell,N}^d|=|V_{\ell,N}^d| = \dim \Pi_{\ell, N}^d.
$$
\end{enumerate}
\end{thm}

If we set $\ell_i = -\k_i -1$ and assume that $\k_i$ are real numbers satisfying $\k_i > -1$, then the polynomials
\eqref{eq:Hahn-simplex} become orthogonal polynomials defined on the lattice in the simplex $V_N^d$, first 
appeared in \cite{KM} and we follow the formulation in \cite{IX07}. It is easy to verify that the formulas of the
orthogonal polynomials and their norms in \cite{IX07} hold if the indices belong to $H_{\ell,N}^d$. In fact, this is how 
the set $H_{\ell,N}^d$  was discovered and defined. The orthogonality in (i) follows as a consequence. The orthogonality 
implies linear independence, so we only need to prove that $|H_{\ell,N}^d| =  |V_{\ell,N}^d|$. This last identity 
amounts to counting lattice points inside the polyhedron $H_{\ell,N}^d$, which turns out to be highly nontrivial, see 
the discussion below, and will be carried out in Sections \ref{se4} and \ref{se5}. 

The Hahn polynomials on the polyhedron can also be described by a family of commuting second order difference operators acting on $H_{\ell,N}^d$. These constructions will be discussed in Section~\ref{se3}. 

\subsection{Admissible polyhedra}\label{ss2.2}

The domain $V_{\ell,N}^d$ is contained in the simplex $T_N^d = \{x \in \RR^d: x_i \ge 0, |x| \le N\}$. The assumption 
\eqref{eq:l-condition} guarantees that the domain touches all faces of the simplex. If this condition is not satisfied, then a 
shift of variables and an adjustment of the values of some indices will change the domain to one that satisfies 
\eqref{eq:l-condition}, as can be seen in Figure 3, in which $\ell_1+\ell_3 = 5 < 6 = N$ and the weight
function is positive precisely on those lattices, represented by bullets, in the polygon.
\begin{figure}[htbp]  
\centering
 \includegraphics[width=2.2in]{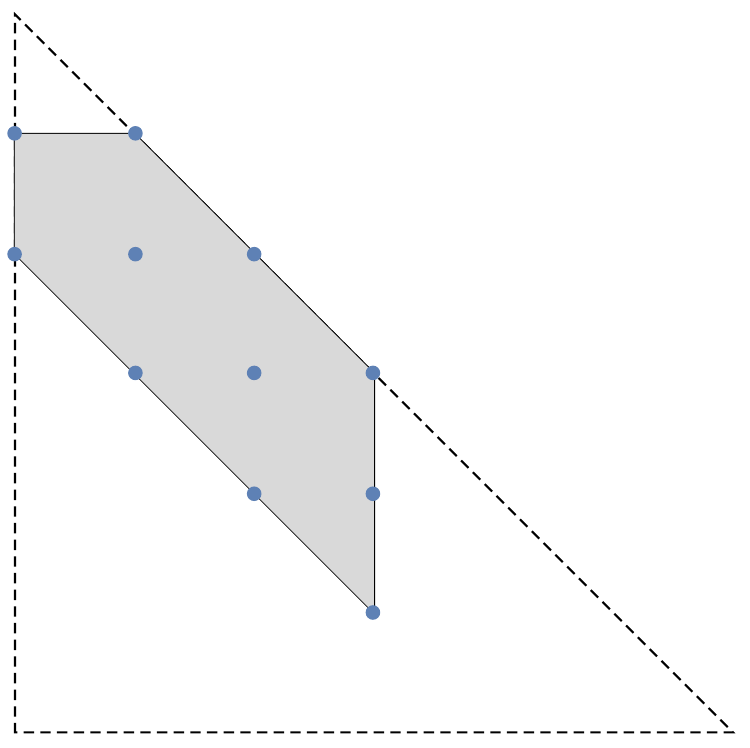}   
 \caption{$\ell_1 = 3$, $\ell_2 =5$, $\ell_3 = 2$ and $N =6$.} 
\end{figure}
It should be noted, however, that the weight function $\sH_{\ell,N}$ is no longer of the same type for the new region. 
As an example, let us assume $\ell_1+\ell_2 \ge N$, $\ell_2+\ell_3 \ge N$ but $\ell_1+\ell_3 < N$. Setting 
$$
\wh N = \ell_1+\ell_3,\quad \wh \ell_1 = \ell_1,\quad \wh \ell_2 = \ell_2 - (N-\wh N), \quad \wh \ell_3 = \ell_3
$$
and changing variables $y_1 = x_1$ and $y_2 = x_2 - (N-\wh N)$, it is easy to see that the domain $V_{\ell,N}^2$ 
becomes $V_{\wh \ell, \wh N}^2$ in $(y_1,y_2)$ and $\wh \ell_i + \wh \ell_j \ge \wh N$, but the weight function 
$\sH_{\ell,N}(x_1,x_2)$ becomes
$$
    \frac{N!}{(-N- \wh \ell_2)_{\hat N}\,(y_2+1)_{N-\wh N}} \frac{(- \wh \ell_1)_{y_1}(- \wh \ell_2)_{y_2}(- \wh \ell_3)_{\wh N- y_1-y_2}}
          { y_1!y_2! (\wh N- y_1-y_2)!},
$$
which contains, besides the multiplicative constant, an additional factor $(y_2+1)_{N-\wh N}$ in comparison with 
$\sH_{\wh \ell, \wh N}(y_1,y_2)$ in \eqref{eq:poly-weight}. 

The orthogonality of the Hahn polynomials on the polyhedron is defined on the lattice points in the polyhedron $H_{\ell,N}^d$, 
which has a rich geometry. In the case of $d =2$, the generic case of $H_{\ell, N}^2$ is a hexagonal domain, when 
$\ell_j+\ell_j > N$ for $1 \le i,j \le 3$, but it can degenerate to a pentagon, when $\ell_i+\ell_j = N$ for one pair of $i,j$, 
or a quadrilateral, when $\ell_i+\ell_j = N$ for two pairs of $i,j$, or a triangle, when $\ell_i+\ell_j = N$ for all three pairs of $i,j$, 
as shown in Figure 4. 

\begin{figure}
\centering
 \includegraphics[width=2in]{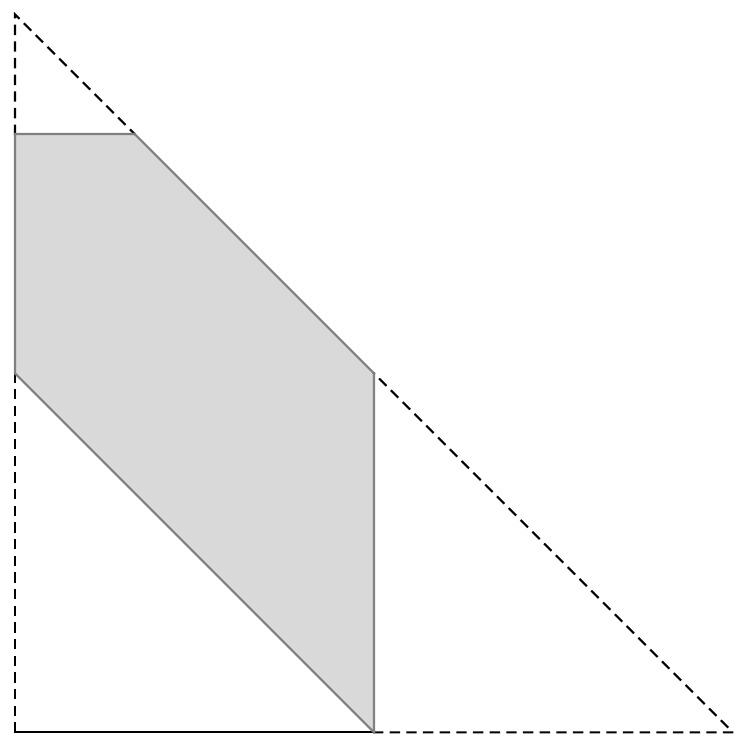}  \includegraphics[width=2in]{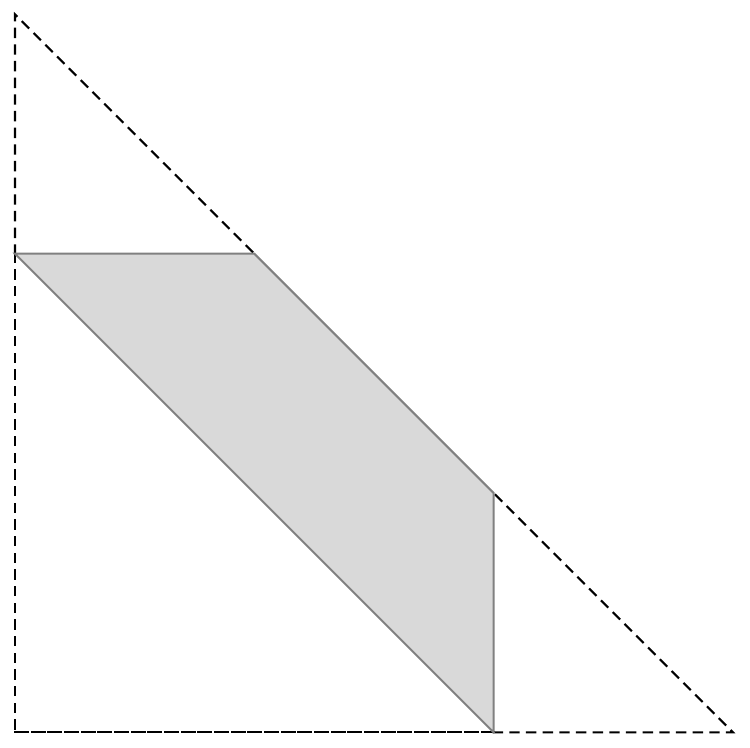}   
 \caption{$(N, \ell_1,\ell_2,\ell_3) = (6,3,5,3)$ (left) and $(6,4,4,2)$ (right)}
\end{figure}

For $d =2$, a very interesting case is when $\ell_1 = \ell_2 = \ell_3 = 2N/3$ for $N = 3 M$ for some 
$M\in \NN$, due to the complete symmetry. The symmetry, however, is best viewed when we consider 
homogeneous coordinates in $(x_1,x_2, x_3)$ with $x _3 = N-x_1-x_2$. In other words, the full symmetry is in the hyperplane 
$\{x \in \RR^3: x_1+x_2+x_3 = N\}$ of $\RR^3$, as shown in  Figure 5.

\begin{figure}[htbp]  
\centering
 \includegraphics[width=2in]{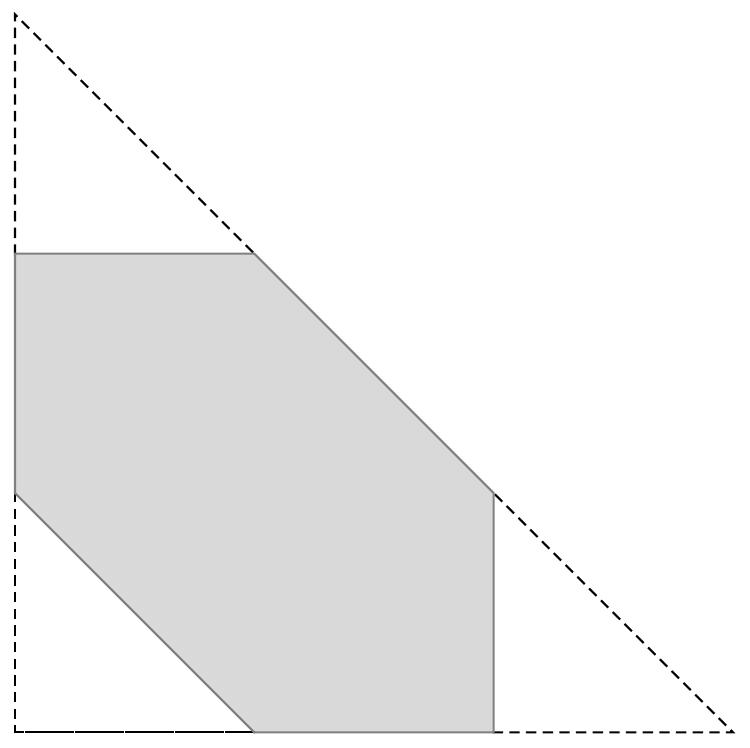}  \includegraphics[width=2.6in]{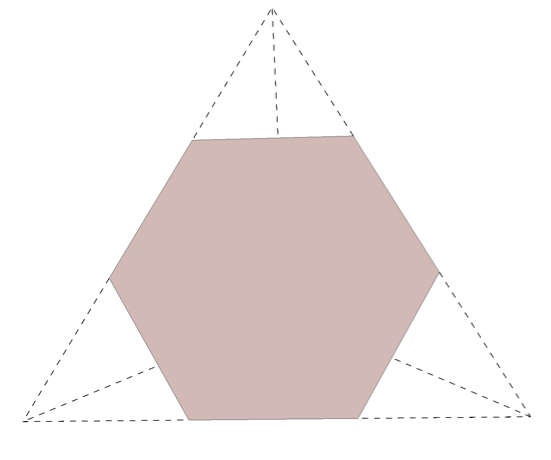}   
 \caption{$(N, \ell_1,\ell_2,\ell_3) = (3M,2M,2M,2M)$}
\end{figure}
 
This last comment also applies to polyhedra in $d \ge 3$. The full symmetry is attained in the variables 
$(x_1, \ldots, x_d, x_{d+1})$ on the hyperplane $x_1+\cdots+x_{d+1} = N$. 
For $d =3$, our polyhedron is a truncated tetrahedron, when
$\ell_i+\ell_j > N$ for all $1 \le i,j \le 4$, which has 4 hexagon faces and 4 triangular faces, and it also include
many degenerated cases, when $\ell_i+\ell_j =N$ for one or more pairs of $i,j$. The most symmetrical cases,
in the hyperplane of $\RR^4$, are truncated tetrahedron and octahedron, as shown in Figure 6. 

\begin{figure}[htbp]  
\centering
 \includegraphics[width=2.2in]{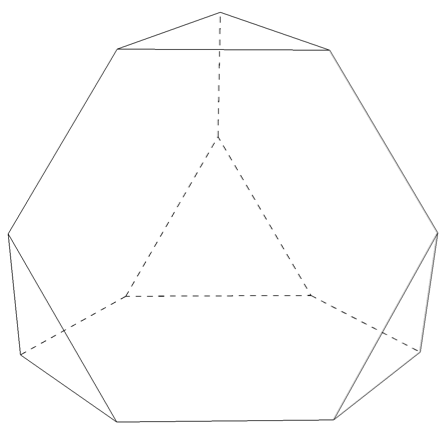} \quad \includegraphics[width=2.2in]{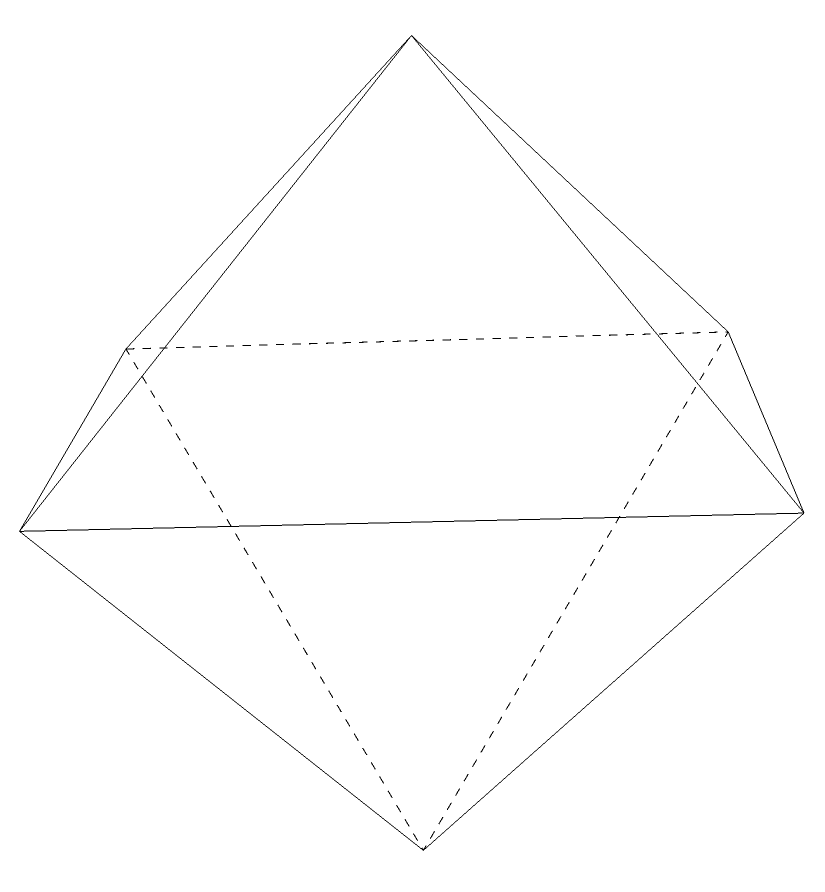}   
 \caption{$ \ell_1=\ell_2=\ell_3=\ell_4 = 2N/3$ (left) and $= N/2$ (right)}
\end{figure}

\subsection{Hahn polynomials with negative integer parameters}\label{ss2.3}
The classical Hahn polynomials $\sQ_n(\cdot;a,b,N)$ are orthogonal with respect to the inner product 
$$
  \la f, g\ra_{a,b} = \sum_{k=0}^N f(x) g(x) w_{a,b}(x) 
$$
when $a, b> -1$, where the weight function is defined on $\NN_N: = \{0, 1, \ldots, N\}$ by 
$$
 w_{a,b} (x) = \frac{(a+1)_x(b+1)_{N-x}}{x!(N-x)!}.
$$
In the previous subsection, we have used the Hahn polynomials with one or both parameters being negative 
integers, which needs clarification. 

If $a + 1 = -\ell$ for an $\ell \in \NN$ and $\ell \le N$, then the weight function  
$$
   w_{-\ell-1,b} (x) = \frac{(-\ell)_x(b+1)_{N-x}}{x!(N-x)!}
$$
is zero if $x > \ell$, so that the sum in the inner product is over $\NN_\ell$.  In this case, however, the weight function 
is no longer positive on the set $\NN_\ell$. By exchanging the positions of $a+1 = -\ell$ and $-N$ in the denominator
of the ${}_3 F_2$ function, it is easy to see that 
\begin{equation}\label{eq:a=-l-1}
      \sQ_n(x; -\ell-1, b, N) =  \sQ_n(x; -N-1, N-\ell+b, \ell),
\end{equation}
which are well-defined for $n =0,1,\ldots, \ell$ and orthogonal with respect to $\la \cdot,\cdot\ra_{-\ell-1,b}$. 

If instead of $a$, we have $b+1 = -\ell$ for an $\ell \in \NN$ and $\ell \le N$, then the weight function
$$
   w_{a,-\ell-1} (x) = \frac{(a+1)_x (-\ell)_{N-x}}{x!(N-x)!}
$$
is zero if $x < N-\ell$, so that the sum in the inner product is over $N-x \in \NN_\ell$. Again, the weight function is 
not positive on its domain of definition. Using the identity 
$$
  \sQ_n(x; a,b,N) = (-1)^n \frac{(b+1)_n}{(a+1)_n} \sQ_n(N-x, b,a,N),
$$
we see that the orthogonal polynomials are given by 
\begin{equation}\label{eq:b=-l-1}
   \sQ_n(x; a,b,N)  = (-1)^n \frac{(-\ell)_n}{(a+1)_n} \sQ_n(N-x, -N-1, N-\ell+a, \ell),
\end{equation}
which are well defined for $n=0,1,\ldots, \ell$. 

If both $a$ and $b$ are negative integers, we do end up with a positive weight function. Let $a = -\ell_1 -1$ and
$b = - \ell_2-1$ for $\ell_i \in \NN$ and $\ell_i \le N$. Assume further that $\ell_1 + \ell_2 \ge N$. Then the weight
function
$$
 (-1)^N w_{-\ell_1-1,-\ell_2-1}(x) = (-1)^N \frac{(-\ell_1)_x (-\ell_2)_{N-x}}{x!(N-x)!}
$$
is well defined and positive for $N-\ell_2 \le x \le \ell_1$. If $\ell_1 \le \ell_2$ then, setting $b+1 = - \ell_2$ in 
\eqref{eq:a=-l-1}, we see that 
$$
 \sQ_n(x; - N-1,N-\ell_1-\ell_2-1,\ell_1)={}_3F_2\left( \begin{matrix} -n, \,  n-\ell_1-\ell_2-1, \, -x \\ -N, \, -\ell_1 \end{matrix}; 1 \right)
$$
are orthogonal polynomials with respect to $\la \cdot ,\cdot\ra_{-\ell_1-1,-\ell_2-1}$. Since $(-1)^Nw_{-\ell_1-1,-\ell_2-1}$ is 
positive on the set
$\{x \in \NN: N-\ell_2 \le x \le \ell_1\}$ of cardinality $\ell_1+\ell_2-N+1$, we see that these orthogonal polynomials are 
well-defined for $n=0,1,\ldots, \ell_1+\ell_2-N$. If $\ell_1 \ge \ell_2$, we  set $a+1= - \ell_1$ in
\eqref{eq:b=-l-1} to conclude that the polynomials
$$
     (-1)^n \frac{(-\ell_2)_n}{(-\ell_1)_n} \sQ_n(N-x; - N-1,N-\ell_1-\ell_2-1,\ell_2)
$$
are well-defined for $n=0,1,\ldots,\ell_1+\ell_2 -N$ and orthogonal with respect to $\la \cdot ,\cdot\ra_{-\ell_1-1,-\ell_2-1}$.

\subsection{Hahn polynomials in several variables}\label{ss2.4}

We now recall  Hahn polynomials of several variables on the simplex $V_N^d$. For $\k_i > -1$, $1 \le i \le d+1$,  they are 
defined for the weight function 
\begin{equation}\label{eq:Hahn-weight}
   \sH_{\k,N} (x) =   \frac{N!}{(|\k|+d+1)_N}\,
       \prod_{i=1}^d \frac{(\k_i + 1)_{x_i}}{x_i!} \frac{(\k_{d+1}+1)_{N-|x|}}{(N-|x|)!},
\end{equation}
which corresponds to $\ell_i = - \k_i -1$ in our setting. These polynomials are given explicitly as in \eqref{eq:Hahn-simplex} 
by formally replacing $\ell_i = -\k_i -1$. Since $\k_i > -1$, its orthogonal polynomials, $\sQ_\nu$, are defined for all $|\nu| \le N$ 
and $\nu \in \NN_0^d$, and they form a basis of $\Pi_N^d$. In this case, the index set of this basis coincides with the set of 
lattice points in $V_N^d$. 

These polynomials are eigenfunctions of a second order difference operator. In \cite{IX07}, all second order difference operators 
of several polynomials that have discrete orthogonal polynomials as eigenfunctions are characterized. The characterization 
includes several distinct families of orthogonal polynomials that are closely related to the Hahn polynomials on $V_N^d$, 
which we describe below.

\subsubsection{} If we make the following substitutions
\begin{align*}
   \k_i & = -\ell_i -1, \qquad 1 \le i \le d, \\
   \k_{d+1} & = |\ell |+ r + \beta, \\
   N & = - \b -1,
\end{align*}
for $\ell_i \in \NN$, $r, \b \in \RR$, $r> 0$ and $\b>0$, then the weight function \eqref{eq:Hahn-weight} can be 
written as
$$
  \sH(x) =   \frac{(r)_{|\ell|}}{(r+\b+1)_{|\ell|}}\,
   \prod_{i=1}^d \frac{(-\ell_i)_{x_i}}{x_i!} \frac{(\b+1)_{|x|}}{(-|\ell|-r+1)_{|x|}},
$$
which is defined for the parallelepiped set 
$$
   V_\ell^d: = \{x \in \NN_0^d: x_i \le \ell_i\}.
$$ 
Furthermore, using the fact that if $ -a -1\in \NN$, then 
$$
     \sQ_n (x;a,b,N) = \sQ_n(x; -N-1,N+a+b+1, -a-1),
$$  
it is easy to verify that, up to unessential constant factors, the polynomials in \eqref{eq:Hahn-simplex} become the orthogonal 
polynomials on $V_\ell^d$ as defined in \cite[Theorem 5.3]{IX07}. 

\subsubsection{} More generally, we can consider the Hahn polynomials on the set
$$
  V_{N,S}^d = V_N^d \cap \{x: x_i \le \ell_i \,\,\hbox{for} \,\, i \in S\},
$$
where $S$ is a nonempty set $S\subset \{1,2,\ldots,d\}$, $\ell_i +1 \in \NN$ and  $\ell_i \le N$. In the case $S=\{1,2,\ldots,d\}$,
we also assume that $\ell_1 + \cdots + \ell_d > N$. The weight function is the same as $\sH_{\k, N}$ with $\k_i = - \ell_i-1$ 
for $i \in S$, which is positive only if $(-\k_i)$ are large numbers for $i$ in the complement of $S$ . This case is discussed in subsection 5.2.2 of \cite{IX07} and the corresponding orthogonal polynomials are given in \eqref{eq:Hahn-simplex} with 
$\k_i = - \ell_i-1$ for $i \in S$.

\subsubsection{}
The weight function \eqref{eq:poly-weight} on a polyhedron corresponds to setting $\k_i$ to be negative integers 
for $i =1,2, \ldots, d+1$. This case has not been studied previously and turns out to be different from all previous cases in a 
fundamental way. For instance, in all previous cases, the set of indices of orthogonal polynomials and the set of
lattice points in the domain coincide. This, however, is no longer the case when all $\k_i$ are negative integers and the 
definition of $H_{\ell,N}^d$ differs substantially from the polyhedron, see Definition \ref{def:Hln}. In fact, the structure of 
$H_{\ell,N}^d$ is so much more complicated in this case, that even the proof that $|H_{\ell,N}^d| = |V_{\ell,N}^d|$ requires a 
substantial effort. The complexity can be viewed in Figure 2 for $d =3$.

\section{Second order difference operators and integrability}\label{se3}

The difference operator and its decomposition as a sum of self-adjoint partial difference operators which preserve the ideal $\CI(V_{\ell,N}^d)$
are introduced in the first subsection.  Algebraic properties of the symmetry operators and the discrete 
extension of the generic quantum superintegrable system are studied in the second subsection. Limiting cases of the 
parameters are discussed in the third subsection. The reduction to the harmonic oscillator is presented in the fourth subsection.

\subsection{Difference operators on $\Pi_{\ell,N}^d$}\label{ss3.1}
We denote by $\{e_1,e_2,\dots,e_d\}$ the standard basis for $\RR^d$, and by $E_i$ and $E_i^{-1}$ the shift operators 
acting on a function $f(x)$  as follows
\begin{align*}
E_if(x)=f(x+e_i) \quad \hbox{and}\quad E_i^{-1}f(x)=f(x-e_i).
\end{align*}
We use also the forward and the backward difference operators defined by
\begin{align*}
&\fsh_i f(x)=f(x+e_i)-f(x)=(E_i-1)f(x),\\
&\bsh_i f(x)=f(x)-f(x-e_i)=(1-E_i^{-1})f(x).
\end{align*}
The operator
\begin{equation}\label{eq:operator1}
\CL=\sum_{1\leq i\neq j\leq d}\alpha_{i,j}(E_iE_j^{-1}-1)+\sum_{i=1}^d\beta_i(E_i-1)+ \sum_{i=1}^d\gamma_i(E_i^{-1}-1),
\end{equation}
where
\begin{align*}
\alpha_{i,j}&\,=x_j(x_i-\ell_i),\\
\beta_{i}&\,=(x_i-\ell_i)(N-|x|),\\
\gamma_{i}&\,=x_i(N-|x|-\ell_{d+1}),
\end{align*}
coincides with operator for the Hahn polynomials on the discrete simplex, upon setting $\ell_i = -\k_i-1$, see \cite{IX07}.

We can introduce a new coordinate
$$x_{d+1}=N-|x|.$$ 
Then, if we think of $\cL$ as an operator acting on functions of $x_1,\dots,x_{d},x_{d+1}$, we need to interpret $E_{i}$ as 
$E_iE_{d+1}^{-1}$ and $E_{i}^{-1}$ as $E_{d+1} E_i^{-1}$. Therefore, the operator $\cL$ can be written in the more 
symmetric form:
\begin{equation}\label{eq:operator2}
   \cL=\sum_{1\leq i\neq j\leq d+1}x_j(x_i-\ell_i)(E_iE_j^{-1}-1). 
\end{equation}
\begin{defn}\label{de3.1}
For $i\neq j\in\{1,\dots,d+1\}$ define
\begin{equation}\label{eq:operatorij}
\cL_{i,j}=x_j(x_i-\ell_i)(E_iE_j^{-1}-1)+x_i(x_j-\ell_j)(E_jE_i^{-1}-1).
\end{equation}
\end{defn}
With the notation above, and using \eqref{eq:operator2}, we see that the operator $\cL$ can be decomposed as follows
\begin{equation}\label{eq:operator3}
\cL=\sum_{1\leq i< j\leq d+1}\cL_{i,j}.
\end{equation}

\begin{prop}\label{pr3.2} 
For $i\neq j\in\{1,\dots,d+1\}$ and $n\in\NN_0$ the following statements hold:
\begin{enumerate}[\rm (i)]
\item The ideal $\CI(V_{\ell,N}^d)$ is invariant under the action of the operator $\cL_{i,j}$; that is, 
$\CL_{i,j} (\CI(V_{\ell,N}^d)) \subset \CI(V_{\ell,N}^d). $
\item  $\cL_{i,j}(\Pi^d_{n})\subset \Pi^d_{n}$, i.e. $\cL_{i,j}:\Pi^d_{n}\to  \Pi^d_{n}$.
\end{enumerate}
\end{prop}

\begin{proof}
(i) The statement follows easily from the definition of $\CI(V_{\ell,N}^d)$. Indeed, if $f(x)$ vanishes on $V_{\ell,N}^d$ and
 if $x\in V_{\ell,N}^d$, then $f(x+e_i-e_j)$ automatically vanishes when $x+e_i-e_j\in V_{\ell,N}^d$, while for 
 $x+e_i-e_j\not\in V_{\ell,N}$, the coefficient $x_j(x_i-l_i)$ is zero.

(ii) Since 
$$
(x_i+1)^{\nu_i}(x_j-1)^{\nu_j}=x_i^{\nu_i}x_j^{\nu_j}+\nu_ix_i^{\nu_i-1}x_j^{\nu_j}-\nu_jx_i^{\nu_i}x_j^{\nu_j-1}
\mod \Pi^d_{\nu_i+\nu_j-2},
$$
it follows that
$$
x_j(x_i-l_i)(E_iE_j^{-1}-1)x_i^{\nu_i}x_j^{\nu_j}=\nu_ix^{\nu_i}x_j^{\nu_j+1}-\nu_jx^{\nu_i+1}x_j^{\nu_j}\mod \Pi^d_{\nu_i+\nu_j}.
$$
If $j=d+1$, we consider the above formulas for $\nu_{j+1}=0$. Clearly, the right-hand side is skew-symmetric in $i$ and $j$ 
and therefore $\cL_{i,j}(x_i^{\nu_i}x_j^{\nu_j})\in \Pi^d_{\nu_i+\nu_j}$.
\end{proof}

\begin{prop}\label{pr3.3}
Let $i\neq j\in\{1,\dots,d+1\}$. Then:
\begin{enumerate}[\rm (i)]
\item The space $\Pi_{\ell,N}^d$ is invariant under the action of the operator $\cL_{i,j}$.
\item The operator $\cL_{i,j}$ is self-adjoint with respect to the inner product \eqref{eq:Hahn-ipd}.
\end{enumerate}
\end{prop}

\begin{proof}
The invariance of  $\Pi_{\ell,N}^d$ follows immediately from Proposition~\ref{pr3.2}. The self-adjointess of  $\cL_{i,j}$ can be 
deduced from Proposition 2.2 in \cite{IX07}.
\end{proof}

\subsection{Algebraic properties of the operators $\cL_{i,j}$ and superintegrability}\label{ss3.2}

Using the defining relation \eqref{eq:operatorij}, one can establish the following proposition by a straightforward computation.

\begin{prop}\label{pr3.4}
The operators $\cL_{i,j}$ satisfy the following commutation relations
\begin{subequations}\label{3.5}
\begin{align}
[\cL_{i,j},\cL_{k,l}]&=0, \text{ if }i,j,k,l \text{ are distinct,}\label{3.5a}\\
[\cL_{i,j},\cL_{i,k}+\cL_{j,k}]&=0, \text{ if }i,j,k \text{ are distinct.} \label{3.5b}
\end{align}
\end{subequations}
\end{prop}

Recall that the Kohno-Drinfeld Lie algebra is the quotient of the free Lie algebra on generators $\cL_{i,j}$, by the ideal 
generated by the relations in \eqref{3.5}, see \cite{AFV}. Therefore, the operators $\cL_{i,j}$ in \eqref{eq:operatorij} define 
a representation of the Kohno-Drinfeld Lie algebra. Equations \eqref{eq:operator3} and \eqref{3.5} imply that 
\begin{equation}\label{3.6}
[\cL,\cL_{i,j}]=0\qquad\text{ for all }\qquad i\neq j\in \{1,2,\dots,d+1\}.
\end{equation}
If we think of $\cL$ as a quantum Hamiltonian, then the operators $\cL_{i,j}$ represent symmetries, or integrals of 
motion for $\cL$. We denote by $\fAd$ the algebra generated by the operators $\cL_{i,j}$ and we will refer to it as the 
{\em symmetry algebra} for $\cL$.

The system above can be considered as an extension of the generic quantum superintegrable system on the sphere 
$\SS^{d}$ with Hamiltonian given in \eqref{1.6}. 
As we discussed in the introduction, the system on the sphere has been extensively studied in the literature. To see the connection with the 
operator $\cL$ in \eqref{eq:operator1}, we introduce new variables
\begin{equation}\label{3.7}
x_k=Ny_k, \quad \text{ for }k=1,\dots,d.
\end{equation}
If we let $N\to\infty$, the operator $\cL$ becomes the Jacobi operator on the simplex:
\begin{equation}\label{3.8}
\hat{\cL}:=\sum_{k=1}^{d}y_k(1-y_k)\frac{\partial^2}{\partial y_k^2} 
 -2\sum_{1\leq k<j\leq d} y_ky_j\frac{\partial^2}{\partial y_k \partial y_j}+\sum_{k=1}^{d}(|\ell| y_k-\ell_k)\frac{\partial }{\partial y_k},
\end{equation}
see for instance \cite[Section~5.3]{GI}. By changing the variables $y_k=z_k^2$ and by applying an appropriate gauge 
transformation, ignoring unessential constant terms and factors, the operator $\hat{\cL}$ can be transformed into the 
Hamiltonian $\mathcal{H}$ in \eqref{1.6} where $b_i=\frac{1}{4}-(\ell_i+1)^2$, see \cite[Section 2]{I17}.

The operator $\hat{\cL}$ was obtained at the beginning of the last century in the monograph \cite{AK} in the case $d = 2$, 
from the differential equations satisfied by the Lauricella functions. Its connection to the superintegrable system on the sphere 
in arbitrary dimension and link to Jacobi polynomials on the simplex was revealed in \cite{KMT}.  

Let $S_{d+1}$ denote the symmetric group consisting of all permutations of $d+1$ symbols. From \eqref{3.5} it is easy to see 
that for every permutation $\sigma\in S_{d+1}$, the Jucys-Murphy elements, defined by  
$$
\cL_{\sigma_1,\sigma_2}, \cL_{\sigma_1,\sigma_3}+\cL_{\sigma_2,\sigma_3}, \cL_{\sigma_1,\sigma_4}
+\cL_{\sigma_2,\sigma_4}+ \cL_{\sigma_3,\sigma_4},\dots, \sum_{j=1}^{d}\cL_{\sigma_j,\sigma_{d+1}},
$$ 
commute with each other and generate a {\em Gaudin subalgebra} of $\fAd$, see \cite{AFV,Fr}. We now focus on the 
permutation $\sigma_0:(1,2,\dots,d+1)\to (d+1,d,\ldots, 2, 1)$ and consider the Gaudin subalgebra $\fGd$ of $\fAd$ generated by 
\begin{equation}\label{3.9}
\cL_{d,d+1}, \cL_{d-1,d}+\cL_{d-1,d+1}, \cL_{d-2,d-1}+\cL_{d-2,d}+\cL_{d-2,d+1},\dots,\sum_{j=2}^{d+1}\cL_{1,j}.
\end{equation}
We shall show that the subalgebra $\fGd$ provides a spectral characterization of the basis 
$\{\sQ_\nu(x; \ell,N):\nu\in H^{d}_{\ell,N}\}$. For this purpose, we introduce the operators
\begin{equation}\label{3.10}
    \cM_{j}:=\sum_{j\leq k< l \leq d+1}\cL_{k,l }, \qquad j=1,2,\dots,d,
\end{equation}
which represent the partial sums of the operators in \eqref{3.9}. In particular, $\cM_{1}=\cL$ and
$$
  \fGd=\RR\langle\cM_{1},\dots,\cM_{d}\rangle.
$$ 

\begin{prop}\label{pr3.5}
The polynomials $\sQ_\nu(\cdot; \ell,N)$ are eigenfunctions of the operators in $\fGd$ and satisfy the spectral equations
\begin{equation}\label{3.11}
\cM_{k}\sQ_\nu(x; \ell,N) =|\bnu^{k}|(|\bell^{k}|-|\bnu^{k}|+1)\sQ_\nu(x; \ell,N), \qquad k=1,\dots,d.
\end{equation}
\end{prop}
\begin{proof}
When $k=1$, the statement follows from the difference equation satisfied by the Hahn polynomials by setting $\ell_j=-\k_j-1$, 
see \cite[Section 5]{IX07} (where $\s$ is used in place of $\k$). For arbitrary $k\in\{1,\dots,d\}$, note that the operator $\cM_{k}$ 
acts only on the variables $x_{k},\dots,x_{d}$ and coincides with the Hahn operator on the $d+1-k$ dimensional simplex, 
with $N$ replaced by $\Nt=N-|\xb_{k-1}|$. Therefore, we can ignore the first $k-1$ terms in the product \eqref{eq:Hahn-simplex}. 
The proof now follows from the fact that, up to a factor independent of $x_{k},\dots,x_{d}$, the product for $j=k,\dots,d$ in 
\eqref{eq:Hahn-simplex} defines again a Hahn polynomial in the variables $x_{k},\dots,x_{d}$ with indices $\nu_{k},\dots,\nu_{d}$ 
and parameters $\ell_k,\dots,\ell_{d+1}$, $\Nt$.
\end{proof}

The next lemma describes an important nonlinear relation satisfied by the operators $\cL_{i,j}$. It will be the key ingredient in the
construction of smaller sets of algebraic generators of $\fAd$ whose cardinality depend linearly on the dimension $d$. The statement
represents an extension of an analogous result discovered in \cite{I17} 
in the Jacobi case and can be proved by a direct computation.

\begin{lem}\label{le3.6}
Let $\{A,B\}=AB+BA$ denote the anticommutator of the operators $A$ and $B$.
If $i,j,k,m$ are distinct indices, then
\begin{equation}\label{3.12}
\begin{split}
&\ell_k(2+\ell_k)\ell_m(2+\ell_m)\cL_{i,j}=
\left \{ [\cL_{j,k},\cL_{k,m}],[\cL_{i,k},\cL_{k,m}] \right\}-\left\{\cL_{k,m},[\cL_{i,k},[\cL_{j,k},\cL_{k,m}]] \right\}\\
&\qquad-2\left\{\cL_{k,m},\cL_{i,k}\cL_{j,m}\right\}+\ell_k\ell_m\left[\cL_{i,k},[\cL_{k,m},\cL_{j,m}]\right]\\
&\qquad+\ell_j\ell_m\{\cL_{i,k},\cL_{k,m}\}-\ell_m(\ell_m + 2)\left\{\cL_{i,k},\cL_{j,k}\right\}-\ell_k(\ell_k + 2)
\left\{\cL_{i,m},\cL_{j,m}\right\}\\
&\qquad+\ell_i\ell_k \left\{\cL_{j,m},\cL_{k,m}\right\}-4\cL_{j,k}\cL_{i,m}+2(-2+\ell_k\ell_m)\cL_{i,k}\cL_{j,m}\\
&\qquad +2\ell_j(1+\ell_k)\ell_m\cL_{i,k}+\ell_j\ell_k(2+2\ell_m+\ell_k\ell_m)\cL_{i,m}\\
&\qquad+\ell_i\ell_m(2+2\ell_k+\ell_k\ell_m)\cL_{j,k}+2\ell_i\ell_k(1+\ell_m)\cL_{j,m}-\ell_i\ell_j\ell_k\ell_m\cL_{k,m}.
\end{split}
\end{equation}
\end{lem}

Since the right-hand side of equation~\eqref{3.12} is generated by the elements  $\cL_{i,k}$, $\cL_{i,m}$, $\cL_{j,k}$, $\cL_{j,m}$, 
$\cL_{k,m}$, we deduce that $\cL_{i,j}$ is generated by these elements:
\begin{equation}\label{3.13}
\cL_{i,j}\in\RR \langle \cL_{i,k}, \cL_{i,m}, \cL_{j,k}, \cL_{j,m}, \cL_{k,m}\rangle.
\end{equation}
As an immediate corollary, we obtain an explicit set of $2d-1$ generators for the symmetry algebra.
\begin{thm}\label{th3.7}
The symmetry algebra $\fAd$ is generated by the set 
\begin{equation}\label{3.14}
\cS=\{\cL_{1,j}:j=2,3,\dots,d\}\cup \{\cL_{i,d+1}:i=1,2,\dots,d\}.
\end{equation}
\end{thm}

\begin{proof}
The proof is trivial when $d=2$. If $d>2$ and $1<i<j<d+1$ we can use the elements $\cL_{1,i}$, $\cL_{1,j}$, $\cL_{i,d+1}$, 
$\cL_{j,d+1}$, $\cL_{1,d+1}$ from $\cS$ to obtain $\cL_{i,j}$ via \eqref{3.12}.
\end{proof}

The weight \eqref{eq:poly-weight} is invariant under the action of the symmetric group $S_{d+1}$ by permuting 
simultaneously the variables $(x_1,\dots,x_{d+1})$ and the parameters $(\ell_1,\dots,\ell_{d+1})$. For $\tau\in S_{d+1}$ and 
$z\in\RR^{d+1}$ we define $\tau\circ z=(z_{\tau(1)},\dots,z_{\tau(d+1)})$. 
We fix the cyclic permutation
\begin{equation*}
   \tau=(1,2,\dots,d,d+1),
\end{equation*}
and denote
\begin{align*}
\sQ_\nu^{+}(x; \ell,N)&: =\tau\circ \sQ_\nu(x; \ell,N)=\sQ_\nu(\tau\circ x; \tau\circ\ell,N),\\
\sQ_\nu^{-}(x; \ell,N)& : =\tau^{-1}\circ \sQ_\nu(x; \ell,N)=\sQ_\nu(\tau^{-1}\circ x; \tau^{-1}\circ\ell,N).
\end{align*}
Since $\sQ_\nu^{+}(x; \ell,N)$ and $\sQ_\nu^{-}(x; \ell,N)$ are simultaneous permutations, by $\tau$ and $\tau^{-1}$, of the orthogonal polynomials defined in 
\eqref{eq:Hahn-simplex}, both are orthogonal polynomials with respect to the inner product \eqref{eq:Hahn-ipd}.
Similarly, we denote by 
\begin{equation}\label{3.15}
\fGd^{+}=\tau \circ \fGd \qquad \text{ and }\qquad \fGd^{-}=\tau^{-1} \circ \fGd
\end{equation}
the Gaudin subalgebras of $\fAd$ obtained by applying $\tau$ and $\tau^{-1}$ to the Gaudin subalgebra $\fGd$. 
Thus $\fGd^{\pm}$ is generated by $\cM^{\pm}_j=\tau^{\pm 1}\circ \cM_j$, $j=1,2\dots,d$. As a simple corollary from 
Theorem~\ref{th3.7}, we can show that the three algebras $\fGd$, $\fGd^{+}$, $\fGd^{-}$ generate $\fAd$.

\begin{prop}\label{pr3.8}
The set
\begin{equation}\label{3.16}
\tilde{\cS}=\{\cM_j:j=1,2,\dots,d\}\cup \{\cM^{+}_j:j=2,\dots,d\}\cup  \{\cM^{-}_j:j=2,\dots,d\}
\end{equation} 
generates $\fAd$.
\end{prop}

\begin{proof}
From equation \eqref{3.10} we see that
\begin{equation*}
\cM_j^{+}=\cM_{j+1}+\sum_{k=j+1}^{d+1}\cL_{1,k},
\end{equation*}
which shows that 
$$\cL_{1,j}=(\cM^{+}_{j-1}-\cM_{j})-(\cM^{+}_{j}-\cM_{j+1}), \text{ for }j=2,\dots,d+1,$$
with the convention that $\cM_{d+1}=\cM_{d+2}=0$. A similar argument yields
$$\cL_{i,d+1}=(\cM_{i}-\cM^{-}_{i+1})-(\cM_{i+1}-\cM^{-}_{i+2}), \text{ for }i=1,\dots,d,$$
which proves the statement by applying Theorem~\ref{th3.7}.
\end{proof}

\begin{rem}
It is perhaps useful to stress that most of the algebraic properties discussed here do not depend on the fact that $\ell_i\in\NN$ and hold for generic real, or even complex values of $\ell_i$. Indeed, the fact that $\ell_{i}\in\NN$ was used only in the proofs of Proposition~\ref{pr3.2}(i) and Proposition~\ref{pr3.3}. If $\ell_i$ are generic real numbers, then we work with the weight \eqref{eq:Hahn-weight} on the simplex $V^{d}_N$ and the Hahn polynomials discussed in subsection~\ref{ss2.4}, which are defined for all $\nu\in V^{d}_N$. In this case, Proposition~\ref{pr3.2} holds with $V^{d}_{\ell,N}$ replaced by $V^{d}_{N}$, Proposition~\ref{pr3.3} holds with $\Pi_{\ell,N}^d$ replaced by $\Pi_{N}^d$ and all other statements hold without changes.
\end{rem}

We can now describe purely algebraically the irreducible finite-dimensional representations of the symmetry algebra $\fAd$. For 
simplicity of the presentation, we assume below that $\ell_j$ are generic negative numbers  
and therefore we work below with the Hahn polynomials $\sQ_\nu(\cdot; \ell,N)$ which are defined for all $|\nu|\leq N$, and which 
are orthogonal with respect to the weight function \eqref{eq:Hahn-weight}, where $\k_j=-\ell_j-1$. For $n\in\NN$ such that 
$n\leq N$, consider the space $\CV_n^d$ consisting of all polynomials of total degree $n$ that are orthogonal to all polynomials 
of total degree at most $n-1$ with respect to the Hahn weight. From Proposition~\ref{pr3.2}(ii) and Proposition~\ref{pr3.3}, it follows that $\CV_n^d$ is 
invariant under the action of the operators in $\fAd$ and a basis of this space is given by $\CB=\{\sQ_\nu(\cdot; \ell,N):|\nu|=n\}$.
 From Proposition~\ref{pr3.5}, we see that the operators $\cM_j$ act diagonally on $\CB$. From \cite[Section 7]{IX17}, we know 
 that the connecting coefficients between $\CB$ and the bases $\CB^{\pm}=\{\sQ_\nu^{\pm}(\cdot; \ell,N):|\nu|=n\}$ of $\CV_n^d$
can be expressed in terms of the multivariable Racah polynomials introduced by Tratnik \cite{Tr}. Using this fact, we can express 
the action of the operators in $\fGd^{\pm}$ in terms of the multivariable Racah operators introduced in \cite{GI}. Combining 
this with Proposition~\ref{pr3.8}, we obtain explicit formulas for the action of all operators in $\fAd$ on $\CB$ in terms of the 
multivariable Racah variables and operators. We can think now of $\CB$ as the basis of an abstract vector space, and we 
have thus obtained an explicit representation of $\fAd$ on this space. We omit the details since they parallel the computations 
in the Jacobi case, and we refer the reader to  \cite{I17} for similar arguments and constructions. 

\subsection{Other discrete quantum superintegrable systems obtained by limits}\label{ss3.3}

In \cite{IX07}, the difference operators $\CL$ that have orthogonal polynomials of several variables as eigenfunctions are 
classified. The operator $\CL$ given in \eqref{eq:operator1} with generic parameters $\ell_i$ is the most general one. The difference operators for other families of orthogonal polynomials can be considered as limiting cases. In this subsection, we show how these limits lead to other discrete quantum superintegrable systems, related to the Krawtchouk, Meixner and Charlier polynomials. They can be considered as discrete extensions of the harmonic oscillator, as we show in the next subsection. Similar limits between continuous quantum superintegrable systems play an important role in the general theory - see for instance \cite{KMP2} where it is shown that all second order superintegrable systems in two dimensions are limiting cases of the generic quantum superintegrable system on the two-sphere with Hamiltonian given in \eqref{1.6} with $d=2$.

It is interesting to compare the constructions here with earlier works. For instance, in \cite{I12} finite-dimensional representations of the Lie algebra $\mathfrak{sl}_{d+1}$ were used to construct Krawtchouk polynomials in several variables and their spectral properties were related to Cartan subalgebras. In this paper, the Krawtchouk polynomials appear naturally within the context of the Kohno-Drinfeld Lie algebra and are characterized as common eigenfunctions for the operators in Gaudin subalgebras. For numerous  applications of the Krawtchouk polynomials see \cite{DG} and the references therein. 

\subsubsection{\it Krawtchouk polynomials} 
Let $0 < p_i < 1$, $1 \le i \le d$, and $|p|<1$. For $\nu \in \NN_0^d$ and $|\nu| \le N$, 
the Krawtchouk polynomials are defined by, for $x\in V_N^d$,  
\begin{align}\label{eq:KrawK}
\sK_\nu(x; p, N) := & \frac{1}{(-N)_{|\nu|}}  \prod_{j=1}^d   (-N+|\xb_{j-1}|+|\bnu^{j+1}|)_{\nu_j} \\
& \times   \sK_{\nu_j}\left( x_j ; \frac{p_j}{1-|\pb_{j-1}|}, N-|\xb_{j-1}|-|\bnu^{j+1}|\right),\notag
\end{align}
where $\sK_n(t;p,N)$ is the Krawtchouk polynomial of one variable and the normalization is as in 
\cite[Proposition~8.1]{IX17}.
As shown in \cite[Section 6.1.1]{IX07}, these polynomials can be deduced from the Hahn polynomials \eqref{eq:Hahn-simplex} 
by setting $\ell_i +1 = -p_i t$ for $1 \le i \le d$ and $\ell_{d+1} +1 = - (1-|p|)t$ and taking the limit $t \to \infty$.
The operator for the Krawtchouk polynomials, as the limit of \eqref{eq:operator1}, can be written as follows:
\begin{equation} \label{Krop}
\mathcal{L}_\sK =  \sum_{1\leq i,j\leq d}(\delta_{i,j}-p_i)x_j \fsh_i\bsh_j  + \sum_{i=1}^d (p_iN-x_i) \fsh_i,
\end{equation}
see \cite[eq. (6.7)]{IX07}. The decomposition \eqref{eq:operator3} of $\cL=\mathcal{L}_\sK$ holds with 
\begin{equation}\label{Krop2}
    \CL_{i,j} = p_i x_j (E_iE_j^{-1}-1)+ p_j x_i (E_jE_i^{-1}-1),
\end{equation}
where $x_{d+1} = N-|x|$ and $p_{d+1} = 1-|p|$. Moreover, equations \eqref{3.5}, \eqref{3.6} follow immediately under the limit, i.e. the operators $\CL_{i,j}$ in \eqref{Krop2} can be regarded as symmetry operators for $\mathcal{L}_\sK$. However, \eqref{3.12} simplifies significantly under the limits and becomes: 
\begin{equation}\label{eq:Kr-relation}
p_kp_m\cL_{i,j}=\left[\cL_{i,k},[\cL_{k,m},\cL_{j,m}]\right]+p_jp_k\cL_{i,m}+p_ip_m\cL_{j,k}-p_ip_j\cL_{k,m}.
\end{equation}
Using the above relation, we see that the symmetry algebra for the Krawtchouk operator $\mathcal{L}_\sK$ generated by the $\binom{d+1}{2}$ operators in \eqref{Krop2}, can also be generated by the $2d-1$ operators in set $\cS$ defined in \eqref{3.14}. Note also that there are no anticommutators  in \eqref{eq:Kr-relation} and therefore, the set $\cS$ generates the symmetry algebra within the Lie theory.
Since $[\cL_{k,m},\cL_{j,m}]=[\cL_{j,k}, \cL_{k,m}]$, equation \eqref{eq:Kr-relation} can be rewritten as 
\begin{equation}\label{eq:Kr-relation2}
p_kp_m\cL_{i,j}=\left[\cL_{i,k},[\cL_{j,k}, \cL_{k,m}]\right]+p_jp_k\cL_{i,m}+p_ip_m\cL_{j,k}-p_ip_j\cL_{k,m}.
\end{equation}

\subsubsection{\it Meixner  polynomials} 
Let $s >0$, $0 < c_i < 1$, $1 \le i \le d$, and $|c|<1$. For $\nu \in \NN_0^d$, the Meixner polynomials are defined by,
for $x\in \NN_0^d$,  
$$
 \sM_{\nu} (x;s,c) = \frac{1}{(s)_{|\nu|}} \prod_{j=1}^d (s+|\xb_{j-1}|+|\bnu^{j+1}|)_{\nu_j} 
 \sM_{\nu_j}\left(x_j; s+|\xb_{j-1}|+|\bnu^{j+1}|, \frac{c_j}{1- |\cb^{j+1}|}\right),
$$
where $\sM_n(t;s,c)$ denotes the Meixner polynomial of one variable. These polynomials can be derived either from 
the Hahn polynomials on the parallelepiped set $V_\ell^d$, discussed in Subsection 2.4.1, by taking an appropriate 
limit as in \cite[Section 6.1.2]{IX07},
or by setting formally $N= -s$ and $p_j = - c_j /(1-|c|)$ in \eqref{eq:KrawK}. The operator for the Meixner polynomials
can be written as follows:
\begin{equation} \label{Mex}
\mathcal{L}_\sM =  \sum_{1\leq i,j\leq d}\left(\delta_{i,j}+ \frac{c_i}{1-|c|} \right)x_j \fsh_i\bsh_j  
   + \sum_{i=1}^d \left(-x_i+\frac{c_i}{1-|c|}s\right) \fsh_i,
\end{equation}
see \cite[eq. (6.10)]{IX07}. The decomposition \eqref{eq:operator3} of $\cL_\sM$ can be easily obtained by changing the parameters in \eqref{Krop2}. For a different construction of Meixner polynomials in several variables and a direct derivation of the corresponding difference operators see \cite{I122}.

\subsubsection{\it Charlier  polynomials} 
Let $a_j >0$, $1 \le j \le d$. For $\nu \in \NN_0^d$, the Charlier polynomials of $d$ variables are defined by 
$$
   \sC_{\nu}(x;a) = \prod_{i=1}^d \sC_{\nu_i} (x_i; a_i), \qquad 
      \sC_{n} (t; s) = {}_2F_0\left(\begin{matrix} -n,- t \\ - \end{matrix};     - \frac{1}s \right),
$$
where $\sC_n(t;s)$ denote the Charlier polynomial of degree $n$ in one variable. These polynomials are limits 
of the Krawtchouk polynomials in \eqref{eq:KrawK},
$$
   \lim_{N\to \infty}  \sK_\nu(x; a/N, N) = \sC_\nu(x,a), 
$$
where $a/N = (a_1/N,\ldots,a_d/N)$. For $d=1$, this relation is classical and the high dimensional formula follows 
readily from \eqref{eq:KrawK}. Taking the same limit in \eqref{Krop} also shows that the operator for the Charlier 
polynomials is of the form  
\begin{equation} \label{Ch-op}
\mathcal{L_\sC} = \sum_{j=1}^{d}x_j \fsh_j\bsh_j + \sum_{i=1}^d (a_i-x_i) \fsh_i. 
\end{equation}
Note that this time the limit breaks the symmetry, because the operators $\cL_{i,j}$ can have a nonzero only when one of the indices is equal to $d+1$. The operator $\mathcal{L_\sC}$ in \eqref{Ch-op} can be written as $\sum_{i=1}^{d}\CL_{i,d+1}$, where 
\begin{equation}\label{Chop1}
\CL_{i,d+1} = a_i (E_i-1) + x_i (E_i^{-1}-1), \qquad \text{ for }i\in  \{1,\dots,d\},
\end{equation}
are obtained by taking the limit $N\to\infty$ from the operators $\cL_{i,d+1}$ in the Krawtchouk case. When $i\neq j \in\{1,\dots,d\}$, we
consider a different limit, $N \CL_{i,j}$, of the Krawtchouk operators, which as $N\to \infty$ become
\begin{equation}\label{Chop2}
    \CL_{i,j} = a_i x_j (E_iE_j^{-1}-1)+ a_j x_i (E_jE_i^{-1}-1), \qquad \text{ for }i\neq j\in \{1,\dots,d\}.
\end{equation}
The operators $\CL_{i,j}$ in equations \eqref{Chop1}-\eqref{Chop2} generate a symmetry algebra for the Charlier operator $\cL=\mathcal{L_\sC}$ in \eqref{Ch-op} and equations \eqref{3.5} hold when the indices belong to the set $\{1,\dots,d\}$. The operators $\cL_{i,d+1}$ commute with each other, i.e
$$[\cL_{i,d+1},\cL_{j,d+1}]=0$$
when $i\neq j\in \{1,\dots,d\}$, and
$$[\cL_{i,j},\cL_{i,d+1}+\cL_{j,d+1}]=0.$$
However, for $i\neq j\in \{1,\dots,d\}$ we have
$$[\cL_{i,d+1},\cL_{i,j}+\cL_{j,d+1}]=[\cL_{i,d+1},\cL_{i,j}]\neq 0.$$
This is due to the fact that we need different scaling factors for the limits in the Krawtchouk operators. This asymmetry also leads to different limits from equations~\eqref{eq:Kr-relation}-\eqref{eq:Kr-relation2}. If we take $i,j,k\in \{1,\dots,d\}$, $m=d+1$, multiply equation~\eqref{eq:Kr-relation2} by $N^2$ and let $N\to \infty$ we obtain the following formula for the Charlier operators:
\begin{equation}\label{eq:Ch-relation}
a_k\cL_{i,j}=\left[\cL_{i,k},[\cL_{j,k}, \cL_{k,d+1}]\right]+a_ja_k\cL_{i,d+1}+a_i\cL_{j,k}-a_ia_j\cL_{k,d+1}.
\end{equation}
In particular, this means that the symmetry algebra for the Charlier operator $\mathcal{L_\sC}$ is again generated by the $2d-1$ operators in the set $\cS$ in \eqref{3.14} consisting of the operators $\cL_{i,d+1}$, $i=1,\dots,d$ in \eqref{Chop1} together with operators $\CL_{1,j}$ defined in \eqref{Chop2}, $j=2,\dots,d$. 

Next, we rewrite some of the above formulas in a form convenient for the limit in the next subsection. Suppose that 
\begin{equation}\label{Cheqc}
a_1=a_2=\cdots=a_d=a,
\end{equation}
and let us introduce new variables via 
\begin{equation} \label{subs}
x_j=a+\sqrt{2a}\,z_j.
\end{equation}
If we substitute \eqref{subs} in  \eqref{Chop1}, we see that the operators $\cL_{i,d+1}$ can be rewritten as follows:
\begin{equation}\label{Chop3}
\cL_{i,d+1}=\frac{1}{2}\left(1+\frac{\sqrt{2}z_i}{\sqrt{a}}\right)(\sqrt{2a}\fsh_i)(\sqrt{2a_i}\bsh_i) -z_i(\sqrt{2a} \fsh_i),
\end{equation}
while for $i\neq j\in \{1,\dots,d\}$ we obtain from \eqref{Chop2}
\begin{equation}\label{Chop4}
\begin{split}
\frac{1}{a}\cL_{i,j}&=a (-\fsh_i \bsh_j-\fsh_j\bsh_i+\fsh_i\bsh_i+\fsh_j\bsh_j)-\sqrt{2a}(z_j\fsh_i\bsh_j+z_i\fsh_j\bsh_i) \\
&\qquad +\sqrt{2a}z_j(\fsh_i-\bsh_j)+\sqrt{2a}z_i(\fsh_j-\bsh_i).
\end{split}
\end{equation}

\subsection{The quantum harmonic oscillator}\label{se3.4}
We show below how the harmonic oscillator and its symmetry algebra can be obtained as a limit from the discrete quantum systems discussed in the previous subsection.
Suppose that the parameters in the Charlier polynomials are equal, i.e. \eqref{Cheqc} holds and let us introduce new variables $z_1,\dots,z_d$ via equation \eqref{subs}.
Note that a forward shift (resp. a backward shift) in the variable $x_j$ corresponds to a forward shift (resp. a backward shift)
by $1/\sqrt{2a}$ in the variable $z_j$. In particular, 
$$
\sqrt{2a}\fsh_{j}f(z_j)=\frac{f(z_j+1/\sqrt{2a})-f(z_j)}{1/\sqrt{2a}}\to \frac{\pd f(z_j)}{\pd z_j} \qquad\hbox{as $a\to\infty$}.
$$
Consequently, it follows readily that 
\begin{equation}\label{L6}
\sqrt{2a}\fsh_{j} \to \frac{\pd }{\pd z_j} \quad\hbox{and}\quad \sqrt{2a}\bsh_{j} \to \frac{\pd }{\pd z_j} 
\qquad\hbox{as $a\to\infty$}.
\end{equation}
Therefore, if we use equations \eqref{L6} and let $a\to \infty$ in \eqref{Chop3}, we obtain the following operators
\begin{equation}\label{HO1}
\cL_{i,d+1}=\frac{1}{2}\frac{\pd^2}{\pd z_i^2}-z_i\frac{\pd}{\pd z_i} \qquad \text{ for }i\in \{1,\dots,d\}.
\end{equation}
Similarly, if we let $a\to \infty$ in  \eqref{Chop4} we obtain
\begin{equation}\label{HO2}
\cL_{i,j}=\frac{1}{2}\left( \frac{\pd}{\pd z_i}- \frac{\pd}{\pd z_j}\right)^2- (z_i-z_j)\left( \frac{\pd}{\pd z_i}- \frac{\pd}{\pd z_j}\right) \qquad\text{ for }i\neq j \in  \{1,\dots,d\}.
\end{equation}

Therefore, as $a\to\infty$ the Hamiltonian operator for the Charlier polynomials becomes
$$
\CL_\sC \to \CL_H=\sum_{i=1}^{d}\cL_{i,d+1}=\frac{1}{2}\sum_{i=1}^{d}\frac{\pd^2}{\pd z_i^2} -\sum_{j=1}^dz_j\frac{\pd}{\pd z_j} \qquad \text{ as }a\to\infty.
$$
Up to a gauge transformation, the operator $\CL_H$ gives the Hamiltonian for the quantum harmonic oscillator:
$$
\hat{\CL}_H =e^{- \|z\|^2/2} \CL_H  e^{\|z\|^2/2}=\frac{1}{2}\sum_{i=1}^{d}\frac{\pd^2}{\pd z_i^2} -\frac{\|z\|^2}{2}+\frac{d}{2}.
$$
This is in agreement with the limit on the polynomial side. Indeed, it is known that 
$$
\lim_{a\to\infty} (2a)^{n/2}\sC_{n}(\sqrt{2a}t+a;a)=(-1)^nH_n(t),
$$
so that  
$$
\lim_{a\to\infty}   \prod_{j=1}^d(2a)^{\nu_j/2} \sC_{\nu_j}(x_j;a) =(-1)^{|\nu|}  \prod_{j=1}^d H_{\nu_j}(z_j).
$$

The above discussion shows that the discrete quantum systems associated with the Krawtchouk or Charlier operators can be regarded as discrete extensions of the quantum harmonic oscillator. The operators $\cL_{i,j}$, $i\neq j\in\{1,\dots,d+1\}$ defined by equations \eqref{HO1}-\eqref{HO2} commute with $\cL_{H}$ and generate a symmetry algebra. We can easily obtain an explicit set of $2d-1$ generators for this algebra under the limit.  Indeed, if we divide \eqref{eq:Ch-relation} by $a^2$ and let $a\to\infty$, we obtain the following relation for the symmetry operators of the harmonic oscillator:
\begin{equation}\label{eq:HO-relation}
\cL_{i,j}=\left[\cL_{i,k},[\cL_{j,k}, \cL_{k,d+1}]\right]+\cL_{i,d+1}+\cL_{j,k}-\cL_{k,d+1}.
\end{equation}
Thus, the symmetry algebra is generated again by the set 
$$\cS=\{\cL_{1,j}:j=2,3,\dots,d\}\cup \{\cL_{i,d+1}:i=1,2,\dots,d\}.$$
When $i\neq j \in  \{1,\dots,d\}$ we can be replace the operators $\cL_{i,j}$ by the simpler operators
\begin{equation}\label{HOD}
\cD_{i,j}=\cL_{i,j}-\cL_{i,d+1}-\cL_{j,d+1}=-\frac{\pd^2}{\pd z_i\pd z_j}+z_i\frac{\pd}{\pd z_j}+z_j\frac{\pd}{\pd z_i} \quad \text{ for }i\neq j \in  \{1,\dots,d\}.
\end{equation}
Equation~\eqref{eq:HO-relation} becomes
\begin{equation}\label{eq:HO-relation2}
\cD_{i,j}=\left[\cD_{i,k},[\cD_{j,k}, \cL_{k,d+1}]\right],
\end{equation}
and the symmetry algebra is generated also by the set
$$\tilde{\cS}=\{\cD_{1,j}:j=2,3,\dots,d\}\cup \{\cL_{i,d+1}:i=1,2,\dots,d\}.$$
Note also that this algebra contains the first-order symmetries for the harmonic oscillator:
\begin{equation}
z_i \frac{\pd}{\pd z_j}- z_j \frac{\pd}{\pd z_i}= [\cL_{i,j},\cL_{i,d+1}]=[\cD_{i,j},\cL_{i,d+1}]  \text{ for }i\neq j \in  \{1,\dots,d\}.
\end{equation}
Finally, under the gauge transformation, we can write the symmetries of $\hat{\CL}_H$ as 
\begin{align*}
& \hat{\cL}_{i,d+1}= e^{- \|z\|^2/2} \cL_{i,d+1}  e^{\|z\|^2/2}=
\frac{1}{2}\frac{\pd^2}{\pd z_i^2}-\frac{z_i^2}{2}+\frac{1}{2}, \qquad \text{ for }i\in \{1,\dots,d\},\\
& \hat{\cD}_{i,j}= e^{- \|z\|^2/2} \cD_{i,j}  e^{\|z\|^2/2}=
-\frac{\pd^2}{\pd z_i\pd z_j}+z_iz_j, \qquad \text{ for }i\neq j \in  \{1,\dots,d\},
\end{align*}
and this algebra is generated by 
$$\hat{\cS}=\{\hat{\cD}_{1,j}:j=2,3,\dots,d\}\cup \{\hat{\cL}_{i,d+1}:i=1,\dots,d\}.$$

It is interesting to compare our reduction to the harmonic oscillator with the work \cite{GVZ} in which the authors take a rather different path and go in the opposite direction  to construct Krawtchouk polynomials in several variables, starting with the harmonic oscillator.

\section{Hahn polynomials on hexagonal domains}\label{se4}
 
In this section we complete the proof of Theorem \ref{thm:OP-polyhedra} for $d =2$ by proving $|H_{\ell,N}^d| = |V_{\ell,N}^d|$ 
for $d =2$. In fact, in this case, we can prove a stronger result. 

Let $\ell_1,\ell_2,\ell_3$ satisfy $\ell_i \le N$ and $\ell_i + \ell_j \ge N$. In this case, the domain 
$$
    V_{\ell,N}^2: = \{(v_1,v_2) \in \NN_0^2: v_1 \le \ell_1, v_2 \le \ell_2, N-\ell_3 \le v_1+ v_2 \le N\}
$$ 
is hexagonal, including its degenerate cases, and it can be regarded as the set obtained from  the triangle 
$V_N^2$ by cutting the three corners. The number of lattice points in $V_{\ell, N}^2$ is given by 
$$
  |V_{\ell, N}^2| = \binom{N+2}{2} - \binom{N-\ell_1+1}{2}- \binom{N-\ell_2+1}{2}- \binom{N-\ell_3+1}{2}. 
$$
The polynomials $\sQ_{\nu_1,\nu_2}(\cdot; \ell,N)$ in \eqref{eq:Hahn-simplex} are well defined and do not vanish 
on $V_{\ell,N}^2$ if $(\nu_1,\nu_2)$ is in the set
\begin{align*}
       H_{\ell,N}^2: = \{ (\nu_1,\nu_2):  0 \le \nu_1\le \ell_1, 0\le \nu_2 \le h_{\ell,N}(\nu_1) -1\}, 
\end{align*}
where 
$$
h_{\ell,N} (\nu_1):= \min\left(\ell_2,\ell_3,\left\lfloor \frac{\ell_2+\ell_3-\nu_1}{2}\right\rfloor,\ell_1+\ell_2+\ell_3-N-\nu_1,
N-\nu_1\right)+1.
$$
The function $h_{\ell,N}$ is the {\it height function} of $H_{\ell,N}^2$, in the sense that the value of $h_{\ell,N}(\nu_1)$ is the 
number of integer points on the vertical line $x_1 =\nu_1$ in $H_{\ell,N}^2$. Similarly, we define the height function $v_{\ell,N}$ for $V^2_{\ell,N}$ by setting $v_{\ell,N}(\nu_1)$ to be the number of integer points on the vertical line $x_1 =\nu_1$ in $V^2_{\ell,N}$, i.e.
$$ v_{\ell,N}(\nu_1)=|V^2_{\ell,N}\cap \{(x_1,x_2):x_1=\nu_1\}|.$$

\begin{thm} \label{thm:d=2shuffle}
There is a permutation $\tau$ of  $\NN_{\ell_1}:=\{0,\ldots, \ell_1\}$, so that  $$v_{\ell,N}\circ \tau = h_{\ell,N}.$$ 
In particular, the
set of orthogonal polynomials
$$
   \Pi_{\ell,N}^2:= \{\sQ_{\nu_1,\nu_2}(\cdot; \ell,N): (\nu_1,\nu_2)\in H^2_{\ell,N}\}
$$  
on the hexagon $V_{\ell,N}^2$ has the dimension $|V_{\ell,N}^2|$. 
\end{thm}

The permutation is evident in Figure~\ref{ex-shuffle}, where the values of the height functions are $7,7,6,6,5,5,4,3$ in $V_{\ell,N}^2$ and 
$5,6, 7,7,6,5,4,3$ in $H_{\ell,N}^2$, respectively.

\begin{figure}[htbp] 
\centering
 \includegraphics[width=2.3in]{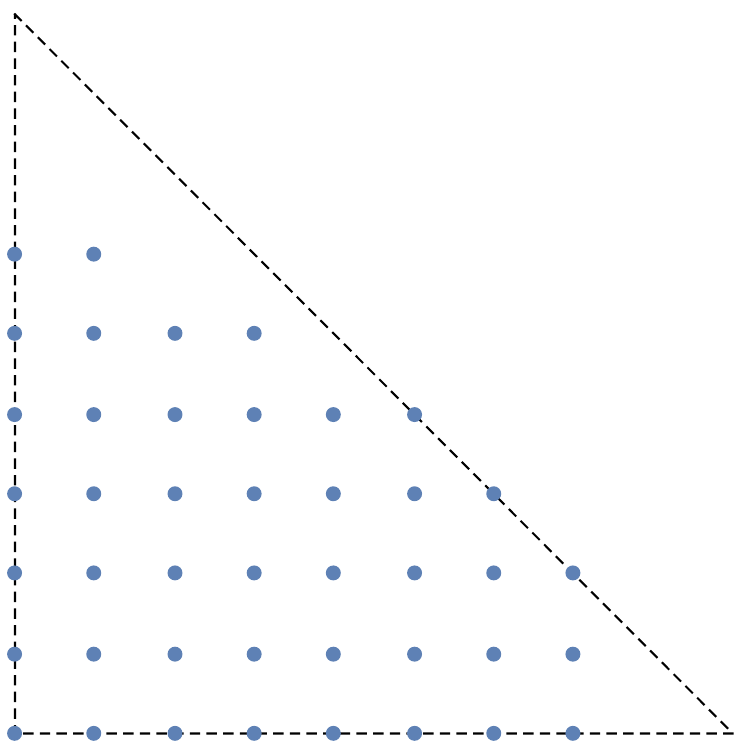}  \quad \includegraphics[width=2.3in]{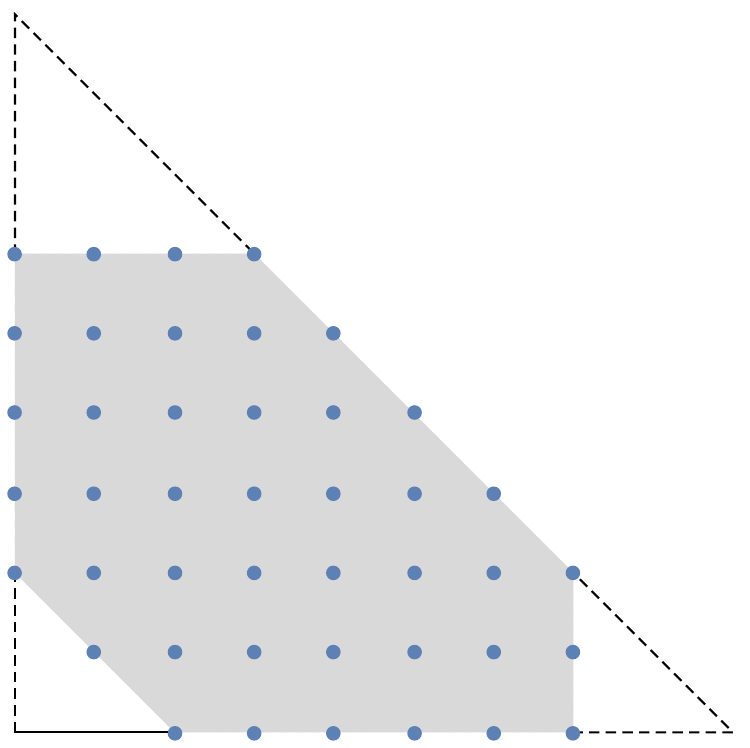}
 \caption{$\ell_1 = 7$, $\ell_2 =6$, $\ell_3 = 7$ and $N =9$. Left: $H^2_{\ell,N}$. Right: $V^2_{\ell,N}$.}
  \label{ex-shuffle}
\end{figure}

The proof follows from the following two lemmas that establish the existence of the permutation, for which we work with 
height functions of the two index sets. 

\begin{lem} \label{lem:cardin1}
The set $\NN_{\ell_1} $ can be partitioned into three disjoint subsets $\NN_{\ell_1} = S_1\cup S_2\cup S_3$,
so that $v_{\ell,N}= \min\{\ell_2,\ell_3\}+1$ on $S_1$ and $|S_1| = |\ell_2 - \ell_3|+1$, 
$v_{\ell,N}$ takes on precisely twice each of the values $\min\{\ell_2,\ell_3\}+1 - i$ for $i =0,1,\ldots, |S_2|/2 -1$ on $S_2$ and  
$|S_2| = 2 \min \{N- \max\{\ell_3,\ell_2\}, \ell_1+\min\{\ell_2,\ell_3\} -N \}$, and 
$v_{\ell,N}$ takes on exactly once each of the values $\max\{\ell_2+\ell_3 - N, N-\ell_1\} -i$ for $i=0,1, \ldots |S_3|-1$ and $|S_3| = 
|2N-\ell_1-\ell_2-\ell_3|$. 
\end{lem}

\begin{proof}
We assume $\ell_3 \ge \ell_2$. The proof for $\ell_3 < \ell_2$ can be handled similarly. Then $v_{\ell,N}(\nu_1) = 
\ell_2+1$ for $\nu_1 \in S_1=\{N-\ell_3, ... ,N-\ell_2\}$. Since the two slanted sides of the hexagon have the same slope,
it is easy to see that $v_{\ell,N}(\nu_1) = \ell_2 +1- i$ if  
$\nu_1 = N-\ell_3 - i$ and $\nu_1 = N-\ell_2+i$ for $i =1,2,\ldots,  \min \{N- \ell_3, \ell_1+\ell_2-N\}$, which corresponds to 
the statement on $S_2$. Taking $S_3 = \NN_{\ell_1} \setminus (S_2\cup S_3)$ proves the result. The case $\ell_3 < \ell_2$
can be handled similarly. 
\end{proof}

We now consider the set $H_{\ell,N}^2$.  

\begin{lem}\label{lem:cardin2}
The conclusion of the Lemma \ref{lem:cardin1} holds if we replace $v_{\ell,N}$ by $h_{\ell,N}$.
\end{lem}

\begin{proof}
The set $H_{\ell,N}^2$, hence the height function $h_{\ell,N}$, is symmetric in $\ell_2$ and $\ell_3$. Assume 
$\ell_3 \ge \ell_2$. If $\nu_1 \le \ell_3-\ell_2$, it is easy to verify that $h_{\ell,N} (\nu_1) = \ell_2+1$, using $\ell_3\le N$ and 
$\ell_1+\ell_3 \ge N$. This defines the set $S_1$ for $H_{\ell,N}^2$. Assume now that $\nu_1 = \ell_3 - \ell_2 + i$. Then
$\nu_2 \le \left\lfloor \frac{2\ell_2- i}{2}\right\rfloor$. We need to consider two cases. If $\ell_1+ \ell_2+\ell_3 \ge 2N$, then
$h_{\ell,N}(\ell_3-\ell_2+ i) =  \left\lfloor \frac{2\ell_2- i}{2}\right\rfloor+1$ if and only if $2 \ell_2 - i \le 2 (N-\ell_3+\ell_2 -i)$, 
or $i/2 \le N-\ell_3$. This implies that 
$$
h_{\ell,N}(\ell_3-\ell_2+ 2j-1) = h_{\ell,N}(\ell_3-\ell_2+ 2j) = \ell_2-j+1
$$ 
for $j=1,2, \ldots, N- \ell_3$. This defines $S_2$ for $H_{\ell,N}^2$ when $N-\ell_3 \le \ell_1+\ell_2 - N$. 
If $\ell_1+ \ell_2+\ell_3 < 2N$, then the displayed equation holds for $j =1,2,\ldots, \ell_1+\ell_2 - N$ since 
$h_{\ell,N}(\ell_3-\ell_2+ i) =  \left\lfloor \frac{2\ell_2- i}{2}\right\rfloor+1$ if and only if $2 \ell_2 - i \le 2 (\ell_1+2\ell_2 -N-i)$, 
or $i/2 \le (\ell_1+\ell_2 - N)$.  This defines $S_2$ for $H_{\ell,N}^2$ when $N-\ell_3 > \ell_1+\ell_2 - N$. 
Finally, $S_3$ for $H_{\ell,N}^2$ is the complement of $S_1\cup S_2$ in $\NN_{\ell_1}$.  
\end{proof}

The definition of the height functions in the two sets clearly extends to higher dimensions. Experimental computations by
a computer algebra system suggest that an analog of Theorem \ref{thm:d=2shuffle} holds at least for $d =3$; that is, 
there is a permutation $\tau$ between the projections of $H_{\ell,N}^3$ and $V_{\ell,N}^3$ onto the plane $x_3=0$, so that if 
$v_{\ell,N}$ is the height function of $V_{\ell,N}^3$, then 
$v_{\ell,N}\circ \tau$ is the height function of $H_{\ell,N}^3$. The proof of such a result, however, looks to be quite difficult, 
as can be seen from Figure 2, and there does not seem to be a simple way to describe the permutation explicitly. 

\section{Lattice points in polyhedra}\label{se5}

Let $\ell_1,\ell_2,\dots,\ell_{d+1}\in\NN_0^{d}$ be such that $\ell_i\leq N$ and $\ell_i+\ell_j\geq N$ for every $i\neq j$. 
We count the number of lattice points in the polyhedron $H_{\bell,N}^d$ when $d \ge 3$. Since we will be working with 
different values of
$\ell_i$ and $N$, we denote
$$
H^d(\ell_1,\ell_2,\dots,\ell_{d+1},N) = H_{\bell,N}^d
$$
in this section. This set is defined by a family of inequalities that we list below for better references: 
\begin{subequations}\label{5.1}
\begin{align}
\nu_j& \leq \ell_j, &&\text{for } j=1,\dots,d,\label{5.1a}\\
\nu_j+2|\bnu^{j+1}| & \leq |\bell^{j+1}|, &&\text{for } j=1,\dots,d,\label{5.1b}\\
|\nu|& \leq  |\ell|-N,\label{5.1c}\\
|\nu|& \leq N.\label{5.1d}
\end{align}
\end{subequations}
We assume $d \ge 3$. Note that in this case $|\ell|\geq  2N$ and therefore \eqref{5.1c} can be ignored, since 
it is implied by \eqref{5.1d}. What we need to prove is part (ii) of Theorem \ref{thm:OP-polyhedra}, which in view of 
\eqref{eq:dimension} amounts to proving the following theorem.

\begin{thm}\label{th4.1}
With notations above we have
\begin{equation}\label{5.2}
|H^d(\ell_1,\ell_2,\dots,\ell_{d+1},N)|=\binom{N+d}{d}-\sum_{k=1}^{d+1}\binom{N-\ell_k+d-1}{d}.
\end{equation}
\end{thm}

The proof of Theorem~\ref{th4.1} occupies the rest of the section and the main ingredients are formulated 
in the next two lemmas. 

\begin{lem}\label{le4.2}
If $k\in\{1,\dots,d-1\}$ and $\ell_k<N$, then 
\begin{equation}\label{5.3}
\begin{split}
&|H^d(N,\dots,N,\ell_{k}+1,\ell_{k+1},\dots,\ell_{d+1},N)|-|H^d(N,\dots,N,\ell_{k},\ell_{k+1},\dots\ell_{d+1},N)|\\
&\qquad=\binom{N-\ell_k+d-2}{d-1}. 
\end{split}
\end{equation}
\end{lem}

\begin{lem}\label{le4.3}
Equation \eqref{5.2} holds when $\ell_1=\ell_2=\cdots=\ell_{d-1}=N$.
\end{lem}

Throughout this section we use the well-known identity:
\begin{equation}\label{5.4}
\binom{a}{k}+\binom{a}{k+1}=\binom{a+1}{k+1}.
\end{equation}
In particular, for $M\in\NN$, this identity leads to
\begin{equation}\label{5.5}
\sum_{\begin{subarray}{c} b\in\NN_0^k\\ |b|\leq M\end{subarray}}\binom{M-|b|+j}{j+1}=\binom{M+j+k}{j+k+1}.
\end{equation}
When $k=0$, we interpret $\NN_0^{0}$ as the trivial space $\{0\}$ and the above sum as a single term corresponding 
to the zero vector $b=0$.

We first show that the proof of Theorem~\ref{th4.1} follows from the lemmas. 

\begin{proof}[Proof Theorem~\ref{th4.1}]
If $\ell_{d},\ell_{d+1},N$ are fixed and $|H^d(\ell_1,\dots,\ell_{d+1},N)|$ is regarded as a function of $\ell_1,\dots,\ell_{d-1}$,
then it is easy to see that this function is uniquely determined by the difference equations in Lemma~\ref{le4.2} together with 
its value at $\ell_1=\cdots=\ell_{d-1}=N$. Using \eqref{5.4} we see that the right-hand side of \eqref{5.2} satisfies the same 
difference equations and has the same value at $\ell_1=\cdots=\ell_{d-1}=N$, thus completing the proof.
\end{proof}

Next we prove the two lemmas. Note that the left-hand side of equation \eqref{5.3} represents the number of points in the set 
\begin{equation}\label{5.6}
\cS_{k}:= H^d(N,\dots,N,\ell_{k}+1,\ell_{k+1},\dots,\ell_{d+1},N)\setminus H^d(N,\dots,N,\ell_{k},\ell_{k+1},\dots\ell_{d+1},N).
\end{equation}
We proceed with the proof of Lemma~\ref{le4.2} by induction on $d$. We show that $|\cS_{k}|$ is equal to the right-hand 
side of equation \eqref{5.3} by considering three cases: 
\begin{enumerate}[i.]
\item $k=1$; 
\item $1<k<d-1$;
\item $k=d-1$.
\end{enumerate}
The first case is relatively simple, because $\ell_1$ appears in only one of the equations in \eqref{5.1} (after we ignore
\eqref{5.1c}). The second case is a bit more involved since $\ell_k$ now appears in two equations, but the essential 
computations come only from the first equation. The last case is the most complicated one, since we can have contributions 
from both equations.

\begin{proof}[Proof of Lemma~\ref{le4.2} when $k=1$]
We determine the number of points in the set $\cS_1=H^d(\ell_1+1,\ell_2,\dots,\ell_{d+1},N)\setminus 
H^d(\ell_1,\ell_2,\dots,\ell_{d+1},N)$. Since $\ell_1$ appears only in equation \eqref{5.1a} with $j=1$, it follows that 
if $\nu\in \cS_1$ then $\nu_1=\ell_1+1$. Moreover, from \eqref{5.1d} it follows that the remaining coordinates must 
satisfy the inequality
\begin{equation}\label{5.7}
  |\bnu^2|\leq N-\ell_1-1.
\end{equation}
It is easy to see now that equations \eqref{5.1a}-\eqref{5.1b} follow from \eqref{5.7} and therefore $\nu\in \cS_1$ if and only if 
$\nu_1=\ell_1+1$ and \eqref{5.7} holds. Thus, $|\cS_1|$ equals the number of solutions of \eqref{5.7}, proving \eqref{5.3} 
when $k=1$.
\end{proof}

\begin{rem}\label{re4.4}
Note that if $l_s=N$ for some $s\in\{2,\dots,d-1\}$ then equation~\eqref{5.1b} for $j=s-1$ can be ignored, because the 
inequality is implied 
by adding equation~\eqref{5.1b} for $j=s$ and equation \eqref{5.1d}. In particular, if $\ell_1=\ell_2=\cdots=\ell_{k-1}=N$ and 
$d\geq 3$, 
the set $H^d(N,\dots,N,\ell_{k},\ell_{k+1},\dots,\ell_{d+1},N)$ consists of all $\nu\in\NN_0^{d}$ such that
\begin{subequations}\label{5.8}
\begin{align}
\nu_j& \leq \ell_j, &&\text{for } j=k,\dots,d,\label{5.8a}\\
\nu_j+2|\bnu^{j+1}| & \leq |\bell^{j+1}|, &&\text{for } j=k-1,\dots,d,\label{5.8b}\\
|\nu| & \leq N.\label{5.8c}
\end{align}
\end{subequations}
\end{rem}

\begin{proof}[Proof of Lemma~\ref{le4.2} when $k\in\{2,\dots,d-2\}$]

From Remark~\ref{re4.4} we know that equations \eqref{5.8} determine $H^d(N,\dots,N,\ell_{k},\dots,\ell_{d+1},N)$. Since 
only two of these equations depend on $\ell_{k}$, namely, equation \eqref{5.8a} with $j=k$ and \eqref{5.8b} with $j=k-1$, 
it follows that if $\nu\in\cS_k$, then we must have equality in at least one of these inequalities, with $\ell_{k}$ replaced 
by $\ell_{k}+1$. We consider these two possible cases next.

\medskip\noindent
{\em Case 1.}  $\nu\in\cS_k$ with $\nu_k=\ell_{k}+1$. Equations \eqref{5.8} for $\nu\in H^d(N,\dots,N,\ell_{k}+1,\ell_{k+1},
\dots,\ell_{d+1},N)$  can be written as follows:
\begin{subequations}\label{5.9}
\begin{align}
\nu_j & \leq \ell_j, &&\text{for } j=k+1,\dots,d,\label{5.9a}\\
(\nu_{k-1}+\ell_k+1)+2|\bnu^{k+1}| & \leq |\bell^{k+1}|, \label{5.9b}\\
(\ell_k+1)+2|\bnu^{k+1}| & \leq |\bell^{k+1}|, \label{5.9c}\\
\nu_j+2|\bnu^{j+1}|& \leq |\bell^{j+1}|, &&\text{for } j=k+1,\dots,d,\label{5.9d}\\
|\bnu_{k-2}|+(\nu_{k-1}+\ell_k+1)+|\bnu^{k+1}|& \leq N.\label{5.9e}
\end{align}
\end{subequations}
Note that equation \eqref{5.9c} can be dropped, since it is implied by \eqref{5.9b}. The last equation shows that 
$|\bnu_{k-2}|\leq N-\ell_k-1$. Fix $\bnu_{k-2}\in\NN_0^{k-2}$ such that $|\bnu_{k-2}|\leq N-\ell_k-1$. We set 
$\nut_1=\nu_{k-1}+\ell_k+1$, $\nut_2=\nu_{k+1}$,\dots, $\nut_{d-k+1}=\nu_{d}$, $\ellt_j=\ell_{j+k-1}$ for 
$j=2,\dots,d-k+2$ and $\Nt=N-|\bnu_{k-2}|$. Then the remaining equations in \eqref{5.9} can be rewritten as
\begin{subequations}\label{5.10}
\begin{align}
\nut_j & \leq \ellt_j, &&\text{for } j=2,\dots,d-k+1,\label{5.10a}\\
\nut_j+2|{\tilde \bnu}^{j+1}| & \leq |{\tilde \bell}^{j+1}|, &&\text{for } j=1,\dots,d-k+1,\label{5.10b}\\
|{\tilde \nu}| & \leq \Nt.\label{5.10c}
\end{align}
\end{subequations}
We need to count the number of the solutions $\tilde \nu\in\NN_0^{d-k+1}$ of the above system for 
which $\nut_1\geq \ell_k+1$. Since $d-k+1\geq 3$, the system of equations \eqref{5.10} represents the points in 
$H^{d-k+1}(\Nt,\ellt_{2},\dots,\ellt_{d-k+2},N)$. Therefore, using the induction hypothesis, we see that the number 
of solutions of \eqref{5.10} for which $\nut_1\geq \ell_k+1$ is given by the formula
\begin{align*}&
|H^{d-k+1}(\Nt,\ellt_{2},\dots,\ellt_{d-k+2},\Nt)|-|H^{d-k+1}(\ell_k,\ellt_{2},\dots,\ellt_{d-k+2},\Nt)|\\
&\quad= \sum_{j=0}^{\Nt - \ell_k-1} \binom{\Nt-j-\ell_k +d-k-1}{d-k}=\binom{N-|\bnu_{k-2}|-\ell_k+d-k}{d-k+1}
\end{align*}
by \eqref{5.5} with $k=1$. Summing over $\bnu_{k-2}$ we see,  by \eqref{5.5}, that the number of $\nu\in\cS_k$ 
for which $\nu_k=\ell_{k}+1$ is 
\begin{equation}\label{5.11}
\sum_{\begin{subarray}{c} \bnu_{k-2}\in\NN_0^{k-2}\\ |\bnu_{k-2}|\leq N-\ell_k-1\end{subarray}}
    \binom{N-|\bnu_{k-2}|-\ell_k+d-k}{d-k+1}=\binom{N-\ell_k+d-2}{d-1}.
\end{equation}

\medskip\noindent
{\em Case 2.}  $\nu\in\cS_k$ and equality holds in \eqref{5.8b} with $\ell_k$ replaced by $\ell_k+1$ when $j=k-1$. 
Equations \eqref{5.8b} under the restriction and \eqref{5.8c} for $\nu\in H^d(N,\dots,N,\ell_{k}+1,\ell_{k+1},\dots,\ell_{d+1},N)$
can be written as follows:
\begin{align*}
\nu_{k-1}+2|\bnu^{k}| & = |\bell^{k}|+1, \\ 
|\bnu_{k-2}|+\nu_{k-1}+|\bnu^{k}| & \leq N.
\end{align*}
The first equation shows that $\nu_{k-1}= |\bell^{k}|+1-2|\bnu^{k}|$ and substituting this into the second equation we obtain 
the inequality $|\bnu^{k}|\geq |\bnu_{k-2}|+|\bell^{k}|-N+1$. Note that $\nu_{k-1}=  |\bell^{k}|+1-2|\bnu^{k}|\geq 0$ if and only 
if $|\bnu^{k}|\leq (|\bell^{k}|+1)/2$ and therefore we must have 
$$
|\bnu_{k-2}|+|\bell^{k}|-N+1 \leq (|\bell^{k}|+1)/2,
$$
which is equivalent to
$$
2|\bnu_{k-2}|+|\bell^{k}|\leq 2N-1.
$$
However, the last inequality cannot hold for any values of $\bnu_{k-2}\in\NN_0^{k-2}$, since $k\leq d-2$ and therefore
$|\bell^{k}|\geq \ell_{d-2}+ \ell_{d-1}+ \ell_{d}+ \ell_{d+1}\geq 2N$. We conclude that there are no solutions $\nu\in \cS_k$
in Case 2. 

\medskip\noindent 
{\em Summary:} We have no solutions in Case 2, so that $|\cS_{k}|$ is given by the right-hand side of equation 
\eqref{5.11} in Case 1, completing the proof of \eqref{5.3} for $k \in \{2,\ldots, d-2\}$.
\end{proof}

\begin{proof}[Proof of Lemma~\ref{le4.2} when $k=d-1$]
When $\ell_1=\cdots=\ell_{d-2}=N$, equations~\eqref{5.8} can be rewritten as follows:
\begin{subequations}\label{5.12}
\begin{align}
\nu_{d-1} & \leq \ell_{d-1}, \label{5.12a}\\
\nu_{d}& \leq \min(\ell_d,\ell_{d+1}),  \label{5.12b}\\
\nu_{d-2}+2(\nu_{d-1}+\nu_{d})& \leq \ell_{d-1}+\ell_{d}+\ell_{d+1},\label{5.12c}\\
\nu_{d-1}+2\nu_{d}&\leq \ell_{d}+\ell_{d+1},\label{5.12d}\\
|\nu|&\leq N.\label{5.12e}
\end{align}
\end{subequations}
If $\nu\in\cS_{d-1}$, we must have equality in \eqref{5.12a} or \eqref{5.12c} with $\ell_{d-1}$ replaced by $\ell_{d-1}+1$.
This leads to two cases but the two cases could have a nonempty overlap. As a result, we consider three cases below. 

\medskip\noindent
{\em Case 1.}  $\nu\in\cS_{d-1}$ with $\nu_{d-1}=\ell_{d-1}+1$. Equations~\eqref{5.12} for $\nu\in H^d(N,\dots,N,
\ell_{d-1}+1,\ell_{d}, \ell_{d+1},N)$ can be written as follows:
\begin{subequations}\label{5.13}
\begin{align}
\nu_{d} & \leq \min(\ell_d,\ell_{d+1}),  \label{5.13a}\\
\nu_{d-2}+2\nu_{d} & \leq \ell_{d}+\ell_{d+1}-\ell_{d-1}-1, \label{5.13b}\\
2\nu_{d} & \leq \ell_{d}+\ell_{d+1}-\ell_{d-1}-1,\label{5.13c}\\
\nu_{d-2}+\nu_{d} & \leq N-|\bnu_{d-3}| -\ell_{d-1}-1 .\label{5.13d}
\end{align}
\end{subequations}
Clearly, the last equation can hold only if $|\bnu_{d-3}|\leq N-(\ell_{d-1}+1)$. Fix $\bnu_{d-3}\in\NN_0^{d-3}$ to satisfy
this condition. We obtain a two-dimensional system for $\nu_{d-2}$ and $\nu_d$. Equations~\eqref{5.13a} and 
\eqref{5.13c} can be ignored since they are implied by \eqref{5.13b}. We then compute the number of solutions of 
\begin{enumerate}[(i)]
\item System I: consisting only of equation \eqref{5.13d}.
\item System II:
\begin{subequations}\label{5.14}
\begin{align}
\nu_{d-2}+2\nu_{d} & \geq \ell_{d}+\ell_{d+1}-\ell_{d-1}, \label{5.14a}\\
\nu_{d-2}+\nu_{d} & \leq N-|\bnu_{d-3}| -\ell_{d-1}-1 .\label{5.14b}
\end{align}
\end{subequations}
\end{enumerate}
The number of solutions of \eqref{5.13b} and \eqref{5.13d} is equal to the number of solutions of System I minus the
number of solutions of the System II. 

For fixed $\bnu_{d-3}$, System I has $\binom{N-|\bnu_{d-3}| -\ell_{d-1}+1}{2}$ solutions. Summing over all possible 
$\bnu_{d-3}$, we obtain $\binom{N-\ell_{d-1}+d-2}{d-1}$ solutions by \eqref{5.5}. 

System II can be rewritten as follows
\begin{equation*}
\frac{\ell_{d}+\ell_{d+1}-\ell_{d-1}-\nu_{d-2}}{2} \leq \nu_d\leq  N-|\bnu_{d-3}| -\ell_{d-1}-1 -\nu_{d-2},
\end{equation*}
which can have a solution if and only if 
\begin{equation*}
\frac{\ell_{d}+\ell_{d+1}-\ell_{d-1}-\nu_{d-2}}{2} \leq N-|\bnu_{d-3}| -\ell_{d-1}-1 -\nu_{d-2},
\end{equation*}
or, equivalently, if and only if $\nu_{d-2}\leq 2N- |\bell^{d-1}| -2-2|\bnu_{d-3}|$. Clearly, if $|\bell^{d-1}|\geq 2N-1$, the 
system will have no solutions. In the case $|\bell^{d-1}|< 2N-1$, the number of solutions of System II is given by the formula
\begin{equation*}
S_1=\sum_{\begin{subarray}{c} \bnu_{d-3}\in\NN_0^{d-3}\\ |\bnu_{d-3}|\leq \lfloor (2N-|\bell^{d-1}|-2)/2\rfloor\end{subarray}}
\sum_{\nu_{d-2}=0}^{2N-|\bell^{d-1}|-2-2|\bnu_{d-3}|}\left\lfloor \frac{2N-|\bell^{d-1}|-2|\bnu_{d-3}|-\nu_{d-2}}{2}\right\rfloor.
\end{equation*}
It is not hard to see that the we can change the limits in the sums and write the last formula as 
\begin{align}\label{5.15}
  S_1=\sum_{\begin{subarray}{c} \bnu_{d-3}\in\NN_0^{d-3}\\
    |\bnu_{d-3}|\leq \lfloor (2N-|\bell^{d-1}|-1)/2\rfloor\end{subarray}}\hat{S}_1(2N-|\bell^{d-1}|-2|\bnu_{d-3}|)
\end{align}
where, for $a\in\NN_0$, 
\begin{align*}
\hat{S}_1(a)&=\sum_{t=0}^{a}\left\lfloor \frac{t}{2}\right\rfloor=\left\lfloor \frac{a}{2}\right\rfloor\left\lfloor 
\frac{a+1}{2}\right\rfloor. \end{align*}
Equation \eqref{5.15} can be used also when $|\bell^{d-1}|= 2N-1$ since it defines $S_1=0$.

Combing the above computations, we obtain the following formula:
\begin{equation}\label{5.16}
\begin{split}
|\cS_{d-1}\cap \{\nu:\nu_{d-1}=\ell_{d-1}+1\}| =\binom{N-\ell_{d-1}+d-2}{d-1} -
 \begin{cases}
  0 & \text{if }|\bell^{d-1}|\geq 2N\\
  S_1 &\text{if }|\bell^{d-1}|< 2N,
\end{cases}
\end{split}
\end{equation}
where $S_1$ is given in \eqref{5.15}.

\medskip\noindent
{\em Case 2.}  $\nu\in\cS_{d-1}$ with $\nu_{d-2}+2(\nu_{d-1}+\nu_{d})= |\bell^{d-1}|+1$. Then 
$\nu_{d-2}=|\bell^{d-1}|-2|\bnu^{d-1}|+1\geq 0$ if and only if $2|\bnu^{d-1}|\leq |\bell^{d-1}|+1$. Together with
the remaining equations in \eqref{5.12} with $\ell_{d-1}$ replaced by $\ell_{d-1}+1$, we obtain the system:
\begin{subequations}\label{5.17}
\begin{align}
\nu_{d-1}&\leq \ell_{d-1}+1, \label{5.17a}\\
\nu_{d}&\leq \min(\ell_d,\ell_{d+1}),  \label{5.17b}\\
\nu_{d-1}+2\nu_{d}&\leq \ell_{d}+\ell_{d+1},\label{5.17c}\\
|\bell^{d-1}|+|\bnu_{d-3}|-N+1 & \leq \nu_{d-1}+\nu_{d}\leq \frac{|\bell^{d-1}|+1}{2}. \label{5.17d}
\end{align}
\end{subequations}
The last equation shows that the system can have a solution only when $|\bell^{d-1}|+|\bnu_{d-3}|-N+1 
\leq \frac{|\bell^{d-1}|+1}{2}$, which is equivalent to $2|\bnu_{d-3}|\leq 2N-1-|\bell^{d-1}|$. In particular, 
if $|\bell^{d-1}|\geq 2N$, the system will have no solutions. If $|\bell^{d-1}|< 2N$, we fix $\bnu_{d-3}\in\NN^{d-3}$ 
such that $2|\bnu_{d-3}|\leq 2N-1-|\bell^{d-1}|$ and compute the number of solutions of
\begin{enumerate}[(i)]
\item System I: consisting of equations \eqref{5.17a}, \eqref{5.17b}, \eqref{5.17c} and  
$\nu_{d-1}+\nu_{d}\leq \lfloor (|\bell^{d-1}|+1)/2\rfloor$,
\item System II: consisting of equations \eqref{5.17a}, \eqref{5.17b}, \eqref{5.17c} and 
$\nu_{d-1}+\nu_{d}\leq |\bell^{d-1}|+|\bnu_{d-3}|-N$.
\end{enumerate}
Then the number of solutions for fixed $\bnu_{d-3}$ is the number of solutions of System I minus the number of 
solutions of System II. 

Both systems are two-dimensional. If we set 
\begin{equation*}
N_1=\left\lfloor \frac{|\bell^{d-1}|+1}{2}\right\rfloor\qquad \text{ and }\qquad N_2=|\bell^{d-1}|+|\bnu_{d-3}|-N,
\end{equation*}
then System I will have $H^2(\ell_{d-1}+1,\ell_d,\ell_{d+1},N_1)$ solutions, while  System II will have $H^2(\ell_{d-1}+1,\ell_d,
\ell_{d+1},N_2)$ solutions. A straightforward computation shows that 
\begin{align*}
&H^2(\ell_{d-1}+1,\ell_d,\ell_{d+1},N_1)-H^2(\ell_{d-1}+1,\ell_d,\ell_{d+1},N_2)\\
&\qquad=(N_1-N_2)(1+|\bell^{d-1}|-N_1-N_2)\\
&\qquad=\left\lfloor \frac{2N+1-|\bell^{d-1}|-2|\bnu_{d-3}|}{2}\right\rfloor 
   \left\lfloor \frac{2N+2-|\bell^{d-1}|-2|\bnu_{d-3}|}{2}\right\rfloor.
\end{align*}
Combing the above computations, we see that:
\begin{equation}\label{5.18}
|\cS_{d-1}\cap \{\nu:\nu_{d-2}+2|\bnu^{d-1}|=|\bell^{d-1}|+1\}|\\
=\begin{cases}
0 & \text{if }|\bell^{d-1}|\geq 2N\\
S_2 &\text{if }|\bell^{d-1}|< 2N,
\end{cases}
\end{equation}
where 
\begin{align}\label{5.19}
S_2 =\sum_{\begin{subarray}{c} \bnu_{d-3}\in\NN_0^{d-3}\\ |\bnu_{d-3}|\leq \lfloor (2N-|\bell^{d-1}|-1)/2\rfloor\end{subarray}}
 \hat{S}_2(2N-|\bell^{d-1}|-2|\bnu_{d-3}|),
\end{align}
where, for $a \in \NN_0$, 
\begin{align*}
\hat{S}_2(a) =\left\lfloor \frac{a+1}{2}\right\rfloor \left\lfloor \frac{a+2}{2}\right\rfloor.
\end{align*}

\medskip\noindent
{\it Case 3}. $\nu\in\cS_{d-1}$ with $\nu_{d-1}=\ell_{d-1}+1$ and $\nu_{d-2}+2(\nu_{d-1}+\nu_{d})= |\bell^{d-1}|+1$.
This is the overlap of Case 1 and Case 2. Clearly, the number of solution is $0$ when $|\bell^{d-1}| \geq 2N$, since
Case 2 has no solutions. If  $|\bell^{d-1}|< 2N$ and if $\bnu_{d-3}\in\NN_0^{d-3}$ is such that 
$2|\bnu_{d-3}|\leq 2N-1-|\bell^{d-1}|$, then (5.17d) with $\nu_{d-1} = \ell_{d-1}+1$ becomes 
$$
 |\bnu_{d-3}|+\ell_{d}+\ell_{d+1}-N\leq \nu_d\leq \left\lfloor \frac{\ell_d+\ell_{d+1} - \ell_{d-1}-1}{2}\right\rfloor.
$$
For fixed $\bnu_{d-3}$, there are $\lfloor (2N-|\bell^{d-1}|-2 | \bnu_{d-3}|+1)/2\rfloor$ solutions. Consequently, 
\begin{equation}\label{5.20}
|\cS_{d-1}\cap \{\nu:\nu_{d-1}=\ell_{d-1}+1,\nu_{d-2}+2|\bnu^{d-1}|=|\bell^{d-1}|+1\}|
=\begin{cases}
0 & \text{if }|\bell^{d-1}|\geq 2N\\
S_3 &\text{if }|\bell^{d-1}|< 2N,
\end{cases}
\end{equation}
where 
\begin{align}\label{5.21}
S_3 =\sum_{\begin{subarray}{c} \bnu_{d-3}\in\NN_0^{d-3}\\ |\bnu_{d-3}|\leq \lfloor (2N-|\bell^{d-1}|-1)/2\rfloor\end{subarray}}
   \hat{S}_3(2N-|\bell^{d-1}|-2|\bnu_{d-3}|), 
\end{align}
where, for $a \in \NN_0$, 
\begin{align*}
\hat{S}_3(a) = \left\lfloor \frac{a+1}{2}\right\rfloor.
\end{align*}

For $a \in \NN_0$, it is easy to see that 
$$
-\hat{S}_1(a)+\hat{S}_2(a)-\hat{S}_3(a) = - \left\lfloor \frac{a}{2}\right\rfloor\left\lfloor \frac{a+1}{2}\right\rfloor
 + \left\lfloor \frac{a+1}{2}\right\rfloor\left\lfloor \frac{a+2}{2}\right\rfloor -\left\lfloor \frac{a+1}{2}\right\rfloor  =0.
$$
Consequently, by \eqref{5.15}, \eqref{5.19} and \eqref{5.21}, 
$$
   -S_1+S_2-S_3=0 \qquad \text{ when }\qquad |\bell^{d-1}| < 2N.
$$
Thus, from equations \eqref{5.16}, \eqref{5.18} and \eqref{5.20} we deduce that
\begin{align*}
|\cS_{d-1}|&=|\cS_{d-1}\cap \{\nu:\nu_{d-1}=\ell_{d-1}+1\}|+|\cS_{d-1}\cap \{\nu:\nu_{d-2}+2|\bnu^{d-1}|=|\bell^{d-1}|+1\}|\\
&\qquad -|\cS_{d-1}\cap \{\nu:\nu_{d-1}=\ell_{d-1}+1,\nu_{d-2}+2|\bnu^{d-1}|=|\bell^{d-1}|+1\}|\\
&=\binom{N-\ell_{d-1}+d-2}{d-1},
\end{align*}
which completes the proof.
\end{proof}

\begin{proof}[Proof of Lemma~\ref{le4.3}]
From Remark~\ref{re4.4} we see that the set $H^d(N,\dots,N,\ell_{d},\ell_{d+1},N)$ consists of all $\bnu\in\NN_0^{d}$ such that
\begin{subequations}\label{5.22}
\begin{align}
\nu_{d} & \leq\min(\ell_d,\ell_{d+1}), \label{5.22a}\\
\nu_{d-1}+2\nu_{d} & \leq \ell_{d}+\ell_{d+1}, \label{5.22b}\\
\nu_{d-1}+\nu_{d} & \leq N-|\bnu_{d-2}|,\label{5.22c}
\end{align}
\end{subequations}
and we need to show that 
\begin{equation}\label{5.23}
|H^d(N,\dots,N,\ell_{d},\ell_{d+1},N)|=\binom{N+d}{d}-\binom{N-\ell_{d}+d-1}{d}-\binom{N-\ell_{d+1}+d-1}{d}.
\end{equation}
Since all equations above are symmetric in $\ell_{d}$ and $\ell_{d+1}$, we can assume that $\ell_{d+1}\leq\ell_{d}$. If we fix 
$\bnu_{d-2}\in\NN_0^{d-2}$ such that $|\bnu_{d-2}|\leq N$ and if we set $\Nt=N-|\bnu_{d-2}|$, then the system is two-dimensional 
and, depending
on the size of $|\bnu_{d-2}|$, it can have one, two, or three of its $\ell_i$ equal to $\Nt$, and the number of solutions of the system 
\eqref{5.22} for $\nu_{d-1}$ and $\nu_d$ is equal to:
\begin{itemize}
\item $|H^{2}(\Nt,\ell_d,\ell_{d+1},\Nt)|$ when $|\bnu_{d-2}|\leq N-\ell_{d}$;
\item $|H^{2}(\Nt,\Nt,\ell_{d+1},\Nt)|$ when $ N-\ell_{d}<|\bnu_{d-2}|\leq N-\ell_{d+1}$;
\item $|H^{2}(\Nt,\Nt,\Nt,\Nt)|$ when $ N-\ell_{d+1}<|\bnu_{d-2}|\leq N$.
\end{itemize}
Therefore, with $\Nt = N - |\bnu_{d-2}|$, 
\begin{align*}
&|H^d(N,\dots,N,\ell_{d},\ell_{d+1},N)|\\
&=\sum_{\begin{subarray}{c} \bnu_{d-2}\in\NN_0^{d-2}\\ |\bnu_{d-2}|\leq N-\ell_{d}\end{subarray}}
  \left(\binom{\Nt +2}{2}-\binom{\Nt -\ell_{d}+1}{2}-\binom{\Nt -\ell_{d+1}+1}{2}\right)\\
&\qquad+\sum_{\begin{subarray}{c} \bnu_{d-2}\in\NN_0^{d-2}\\ N-\ell_{d}<|\bnu_{d-2}|\leq N-\ell_{d+1}\end{subarray}}
\left(\binom{\Nt +2}{2}-\binom{\Nt -\ell_{d+1}+1}{2}\right)\\
&\qquad+\sum_{\begin{subarray}{c} \bnu_{d-2}\in\NN_0^{d-2}\\ N-\ell_{d+1}<|\bnu_{d-2}|\leq N\end{subarray}}
\binom{\Nt +2}{2}\\
&=\sum_{\begin{subarray}{c} \bnu_{d-2}\in\NN_0^{d-2}\\ |\bnu_{d-2}|\leq N\end{subarray}} 
\! \binom{\Nt +2}{2}
\!\! -\sum_{\begin{subarray}{c} \bnu_{d-2}\in\NN_0^{d-2}\\ |\bnu_{d-2}|\leq N-\ell_{d+1}\end{subarray}}
\!\!\! \binom{\Nt -\ell_{d+1}+1}{2} 
\!\! -\sum_{\begin{subarray}{c} \bnu_{d-2}\in\NN_0^{d-2}\\ |\bnu_{d-2}|\leq N-\ell_{d}\end{subarray}}
\!\!\! \binom{\Nt -\ell_{d}+1}{2},
\end{align*}
which gives \eqref{5.23}, by \eqref{5.5}, and completes the proof.
\end{proof}

\end{document}